\documentclass[11pt,letterpaper]{amsart}
\usepackage{enumerate}
\usepackage{url}
\usepackage{verbatim}
\usepackage{cite}
\usepackage{comment}
\usepackage{graphicx}
\usepackage{amsfonts, amsmath, amssymb, amscd, amsthm, bm, cancel}
\usepackage[
linktocpage=true,colorlinks,citecolor=magenta,linkcolor=blue,urlcolor=magenta]{hyperref}
\usepackage{tikz}
\usetikzlibrary{arrows,shapes,trees,backgrounds}
\usepackage{colortbl}
\usepackage{xcolor}

\usepackage[margin=1.1in]{geometry}
\newtheorem{lemma}{Lemma}[section]
\newtheorem{corollary}{Corollary}[section]
\newtheorem{proposition}[lemma]{Proposition}
\newtheorem{theorem}[lemma]{Theorem}
\theoremstyle{definition}

\newtheorem{remark}[lemma]{Remark}
\theoremstyle{definition}
\newtheorem*{ack}{Acknowledgments}
\allowdisplaybreaks
\numberwithin{equation}{section}
\begin{document}
\title[The isoperimetric problem in Riemannian manifolds]{The isoperimetric problem in the Riemannian manifold admitting a non-trivial conformal vector field}
\author[J. Li]{Jiayu Li}
\author[S. Pan]{Shujing Pan}

\address{School of Mathematical Sciences, University of Science and Technology of China, Hefei 230026, P.R. China}

\email{\href{mailto:jiayuli@ustc.edu.cn}{jiayuli@ustc.edu.cn}}
\email{\href{mailto:psj@ustc.edu.cn}{psj@ustc.edu.cn}}

\date{\today}
\keywords{conformal vector fields, isoperimetric inequality, mean curvature type flow}
\subjclass[2020]{53C18, 53C21, 53E10}
\begin{abstract}
In this article, we will study the isoperimetric problem by introducing a mean curvature type flow in the Riemannian manifold endowed with a non-trivial conformal vector field. This flow preserves the volume of the bounded domain enclosed by a star-shaped hypersurface and decreases the area of hypersurface  under certain conditions. We will prove the long time existence and convergence of the flow. As a result, the isoperimetric inequality for such a domain is established. Especially, we solve the isoperimetric problem for the star-shaped hypersurfaces in the Riemannian manifold endowed with a closed, non-trivial conformal vector field, a wide class of warped product spaces studied by Guan, Li and Wang is included.
\end{abstract}

\maketitle
\bibliographystyle{amsplain}
\tableofcontents

\section{Introduction}
The isoperimetric problem is one of the most famous problems in geometry with long history. 
Many methods  were introduced to study the isoperimetric problem. For example, ABP technique, optimal transport, Knothe mapping and so on. Recently, Brendle \cite{brendle2021isoperimetric} obtained a sharp isoperimetric inequality for minimal submanifolds in Euclidean space of codimension at most 2 by proving the Sobolev inequality on submanifolds.
It has been proven that the hypersurface flow is a valid tool to study the isoperimetric problem in manifolds. Gage and Hamilton \cite{gage1986heat} proved the isoperimetric inequality for convex planar domains using the curve shortening flow.  Huisken \cite{huisken1987volume} introduced the volume preserving mean curvature flow and proved the isoperimetric inequality for closed, uniformly convex hypersurfaces in $\mathbb{R}^{n+1}$. Applying the volume preserving mean curvature flow in \cite{huisken1987volume}, Huisken and Yau \cite{huisken1996definition} showed that if $M$ is a $C^4$-asymptotic flat 3-manifold, then  the complement of a compact subset of $M$ admitting a foliation by a family of stable constant mean curvature spheres. Moreover, the leaves of this foliation are the unique stable CMC spheres within a large class of surfaces. After that, their uniqueness result was strengthened in the important work by Qing and Tian \cite{qing2007uniqueness}.  Eichmair and Metzger \cite{eichmair2013large} confirmed that the surfaces found in \cite{huisken1996definition,qing2007uniqueness} are in fact isoperimetric surfaces (conjectured by Bray \cite{bray1997penrose} first). Various extensions of these results have been proven in \cite{brendle2014large,chodosh2021isoperimetry,chodosh2022global,yu2023isoperimetry,eichmair2013unique,neves2009existence} and so on.
Besides the volume preserving mean curvature flow, there are other mean curvature type flows which can be used to study the isoperimetric problem. Schulze \cite{schulze2008nonlinear} proved the isoperimetric inequality in $\mathbb{R}^{n+1}$ with $n\leq7$ using a nonlinear mean curvature type flow.  Guan and Li \cite{guan-li15} constructed a new mean curvature type flow  to prove the isoperimetric inequality for star-shaped hypersurfaces in space forms. After that, Guan, Li and Wang \cite{Guan-Li-Wang2019} extended the mean curvature type flow to a class of warped product spaces and proved the  isoperimetric inequality for star-shaped hypersurfaces under certain conditions. Hypersurface flows can also be used to prove many other geometric inequaities, such as Alexandrov-Fenchel inequalities, Minkowski type inequalities and so on. See for instance
\cite{chen2022alexandrov,ge2014hyperbolic,hu2022locally,li2014geometric,wang2014isoperimetric,chen2022fully,guan2021isoperimetric,brendle-Guan-Li} among many others.

In this paper, we study the isoperimetric problem in a Riemannian manifold admitting a non-trivial conformal vector field. Let $(M^{n+1},\bar{g})$ be a Riemannian manifold, a vector field $\xi$ on $M$ is called a conformal vector field if it satisfies
\begin{align}
	L_{\xi}\bar{g}=2\varphi\bar{g},
\end{align}
where $L$ is the Lie derivative and $\varphi$ is given by $\varphi=\frac{div\xi}{n+1}$. Such vector fields can be described from conformal transformation and remain  to be conformal under any conformal change of metric on $M$. If $\varphi=0$, then the vector field $\xi$ is well known as a Killing vector field and the transformation generated by $\xi$ is an isometry. If $\varphi=c$ is a constant, then the vector field $\xi$ is called a homothetic conformal vector field and the transformation generated by $\xi$ is a homothety. In addition, another special class of conformal vector fields is the closed conformal vector field, which means that $\xi$ satisfies
\begin{align}
	\overline{\nabla}_X\xi=\varphi X,\ \forall X\in TM.
\end{align} 
Riemannian manifolds endowed with closed, non-trivial conformal vector fields are closely related to warped product spaces \cite{chen2017differential}  with a 1-dimensional factor.  It's well known that a Riemannian manifold $(N,\bar{g})$ with the warped product structure $\bar{g}=dr^2+\phi^2(r)\tilde{g}$ is naturally equipped with a closed conformal vector field $\xi=\phi(r)\partial_r$, where $r\in\mathbb{R}$, $\phi=\phi(r)$ is a smooth function defined on $r$ and $\tilde{g}$ is a metric of a complete Riemannian manifold. Actually, $\xi=\phi(r)\partial_r$ is the gradient of a function at this time and such a vector field is called a gradient conformal vector field. Reversely, it has been proven that any Riemannian manifold endowed with a closed, non-trivial conformal vector field is locally isometric to a warped product with a 1-dimensional factor \cite{Montiel99}. There are some rigidity results on the manifolds endowed with non-trivial conformal vector fields established,  see Yano \cite{yano1940concircular} , Obata \cite{obata1962certain}, Lichnerowicz \cite{lichnerowicz1964transformations}, Goldberg and Kobayashi \cite{goldberg1962conformal} for instance, and which can also be found in the book \cite{Yano1970}.
 
One of our motivations to consider the isoperimetric problem in such manifolds is the existence of the totally umbilical hypersurface. Let $\mathcal{Z}(\xi)$ be the set consists of zero points of the vector field $\xi$ , it is also called the singular set. We can conclude that $\mathcal{Z}(\xi)$ is a discrete  set under some condition (see \S \ref{preliminaries}).  Let $M'=M\setminus\mathcal{Z}(\xi)$ be the open dense set  and $\mathcal{F}(\xi)$ be the foliation in $M'$ determined by the $n$-dimensional distribution 	\[p\in M'\rightarrow\mathcal{D}(p)=\{v\in T_pM|\langle\xi(p),v\rangle=0\}.\] 
When $\xi$ is closed,  the distribution $\mathcal{D}$ is integrable on $M'$. Montiel\cite{Montiel99} proved that each connected leaf of $\mathcal{F}(\xi)$ is a totally umbilical hypersurface with constant mean curvature. Moreover, he established that these leaves are the only hypersurfaces with constant mean curvature in many cases.  He also proved that these leaves are stable under extra condition on Ricci curvature (\cite{montiel98}).  Actually, if we assume the distribution $\mathcal{D}$ is integrable, each connected leaf of $\mathcal{F}(\xi)$ is totally umbilical with mean curvature $H=n\frac{\varphi}{\vert\xi\vert}$ even if $\xi$ is not closed (see Proposition \ref{umbilical}).  When we assume further that these leaves are compact hyperurfaces with constant mean curvature, i.e $\frac{\varphi}{\vert\xi\vert}$ is constant along any leaf. It seems reasonable to consider them as potential candidates for  the isoperimetric problem's solutions in such a Riemannian manifold.  We will discuss this problem carefully in \S \ref{sol to isoper}.

 In the manifold endowed with a non-trivial conformal vector field, a hypersurface $\Sigma$ embedded in $M'$ is called a star-shaped hypersurface if there exists a unit normal vector field $\nu$ of $\Sigma$ such that the support function $u=\langle\xi,\nu\rangle>0$. We call  $\nu$  the unit outward normal vector field. We shall consider the following mean  curvature type flow for star-shaped hypersurfaces in $(M^{n+1}, \bar{g})$:
  \begin{align}\label{flow-0}
 	\frac{\partial F}{\partial t}=(n\varphi-uH)\nu,
 \end{align}
where $F$ is the embedded map, $H$ is the mean curvature, $\nu$ is the unit outward normal vector field and $u=\langle\xi,\nu\rangle$ is the support function of the hypersurface. 
In the case that $M$ is a space form, the flow was introduced by Guan and Li \cite{guan-li15}.   Guan, Li and Wang \cite{Guan-Li-Wang2019} studied the case that $(M,\bar{g})=(\mathbb{R}\times_{\phi}\mathbf{B}^n,dr^2+\phi^2(r)\tilde{g})$ is the warped product spaces, where $(\mathbf{B}^n,\tilde{g})$ is a closed Riemannian manifold.
Note that $\varphi=\phi'(r)$ in such warped product spaces, and especially in space forms we have $(B^n,\Tilde{g})$ is the standard sphere $S^n$ in $\mathbb{R}^{n+1}$ and
\begin{equation*}
	\phi(r)=\left\{
	\begin{aligned}
		&r,\  r\in[0,+\infty)\ \text{when} \ M=\mathbb{R}^{n+1} \\
		&\sin r,\ r\in[0,\pi) \ \text{when} \ M=\mathbb{S}^{n+1}\\
		& \sinh r, \ r\in[0,+\infty)\ \text{when} \ M=\mathbb{H}^{n+1}.
	\end{aligned}
	\right.
\end{equation*}
 The important feature of flow (\ref{flow-0}) is that along the flow, volume of the enclosed domain is constant and the area of hypersurface is monotonically decreasing under some conditions. This property yields that, if the flow converges, the solution hypersurfaces would converge to a solution of the isoperimetric problem. As Guan and Li pointed out in \cite{guan-li15}, the key of proving the monotone property is  Minkowski identities. We will derive the  Minkowski identities on hypersurfaces of the manifold endowed with a non-trivial conformal vector field (see Lemma \ref{lemma-minikowsi}). Then, if we assume that any totally umbilical leaf of the foliation $\mathcal{F}(\xi)$ has constant mean curvature, we can prove the monotone property (see Theorem \ref{monotonethm}).
 
One of the main results of this article is the following  convergence theorem:
\begin{theorem}\label{convergence-0}
	Assume that $(M^{n+1},\bar{g})$ is a Riemannian manifold endowed with a complete conformal vector field satisfying the following conditions 
	\begin{itemize}
		\item[(i)]$\varphi>0$;
		\item[(ii)] $\varphi^2-\xi(\varphi)>0$;
		\item[(iii)] any connect leaf of totally umbilical Riemannian foliation $\mathcal{F}(\xi)$ on $M'$ is a closed hypersurface with  constant mean curvature and it can be represent by the level set $\{\frac{\vert\xi\vert}{\varphi}=r\}$ for some constant $r>0$;
		\item[(iv)]the direction determined by $\xi$ is of least Ricci curvature on $M$, that is
		\begin{align}\label{riccicond-0}
			\vert\xi\vert^2\overline{Ric}(X,X)-\vert X\vert^2\overline{Ric}(\xi,\xi)\geq 0, \ \forall X\in TM.
		\end{align}
	\end{itemize}
	Let $\Sigma_0$ be a star-shaped, closed hypersurface embedded in $M'$. Then the evolution equation (\ref{flow-0}) with $\Sigma_0$ as a initial data has a smooth solution for $t\in[0,+\infty)$. Moreover, the solution hypersurfaces converge to a  totally umbilical hypersurface $\Sigma_{\infty}$ whose unit normal vector field $\nu_{\infty}$ attains least Ricci curvature on $M'$, that means
	\begin{align*}
		\overline{Ric}(\nu_{\infty},\nu_{\infty})=\overline{Ric}(\mathcal{N},\mathcal{N}),
	\end{align*}
 where $\mathcal{N}=\frac{\xi}{\vert\xi\vert}$ is the unit vector field. 
\end{theorem}
\begin{remark}
	\begin{itemize}
		\item[(1)] By Obata's result \cite{obata70}, the condition $\varphi>0$ is sufficient to ensure that singular points are isolated. The following identity
		\[	\overline{\nabla}\frac{\vert\xi\vert^2}{\varphi^2}= 2\frac{\varphi^2-\xi(\varphi)}{\varphi^3}\xi\]
		holds (see Prop \ref{conformal vec prop1}), which means the level sets of the function $\frac{\vert\xi\vert^2}{\varphi^2}$ on $M'$ are hypersurfaces oriented by $\mathcal{N}$ under the conditions $\varphi>0$ and $\varphi^2-\xi(\varphi)>0$. Moreover, $\frac{\vert\xi\vert^2}{\varphi^2}$ is strictly increasing along the direction $\xi$. Hence, $\mathcal{Z}(\xi)$ contains at most one point.
		\item[(2)] The condition $\varphi^2-\xi(\varphi)>0$ implies that the mean curvatures of leaves of $\mathcal{F}(\xi)$ are strictly decreasing along the direction $\xi$ (see (1) in Remark \ref{C0remark}). If $M$ is a warped product space $(M,\bar{g})=(\mathbb{R}\times_{\phi}\mathbf{B}^n,dr^2+\phi^2(r)\tilde{g})$, then $\varphi^2-\xi(\varphi)=(\phi')^2-\phi''\phi$.
		\item[(3)] The condition (iv) on the Ricci curvature is needed for the monotone property of the flow. If $\xi$ is closed, this condition implies that the compact leaf of $\mathcal{F}(\xi)$ is a stable CMC hypersurface (see \cite{montiel98}).
		\item[(4)] To prove the existence and convergence of the flow, Guan and Li \cite{guan-li15}(see also \cite{Guan-Li-Wang2019}) considered a equivalent equation of the radial function in warped product spaces. In our case, loss of such a radial function forces us to control the second fundamental form and its higher derivatives. Fortunately, we discover that the second fundamental form can be generated by derivatives of the function $\frac{\vert\xi\vert^2}{\varphi^2}$. Thus, we only need to obtain a priori estimates of the function $\frac{\vert\xi\vert^2}{\varphi^2}$ along the flow.
	\end{itemize}
\end{remark}
Next, we are attempted to prove the limit hypersurface $\Sigma_{\infty}$ of the solution is a leaf of the foliation $\mathcal{F}(\xi)$, then combining with the monotonicity in Theorem \ref{monotonethm}, we can obtain the isoperimetric inequality for star-shaped hypersurfaces in $M$ as a conclusion. Let $S(r)=\{p\in M|\frac{\vert\xi\vert}{\varphi}(p)=r\}$ be the level set of the function $\frac{\vert\xi\vert}{\varphi}$. Given any $r_1\geq 0$, assume that $\Sigma\subset M'$ is a  closed, star-shaped hypersurface satisfies $\min_{\Sigma}\frac{\vert\xi\vert}{\varphi}>r_1$, then there exists a bounded domain $\Omega$ enclosed by $\Sigma$ and $S(r_1)$.  For any $r>r_1$, denote $B(r)$ the bounded domain enclosed by $S(r)$ and $S(r_1)$. If $M$ has a singular point, then $r_1$ can be equal to 0.

If the Ricci curvature condition (iv) in Theorem \ref{convergence-0} is strict, which means the $\mathcal{N}$ direction is the only direction of least Ricci curvature. Then $\Sigma_{\infty}$ is sure a leaf of $\mathcal{F}(\xi)$ and the following isoperimetric inequality is a directly corollary of Theorem \ref{convergence-0}.
\begin{theorem}\label{isoineq1}
Assume that $(M^{n+1},\bar{g})$ is a Riemannian manifold endowed with a complete conformal vector field satisfying all the conditions in Theorem \ref{convergence-0}. Moreover, we assume extra that our hypothesis on the Ricci curvature of $M$ is strict. Let $\Sigma\subset M'$ be a star-shaped, closed hypersurface which satisfies $\min_{\Sigma}\frac{\vert\xi\vert}{\varphi}>r_1$, then 
		\begin{align}\label{isoineq-in1}
			Area(\Sigma)\geq Area(S(r^*)),
		\end{align}
	where $r^*$ is the unique real number determined by $\text{Vol}(B(r^*))=\text{Vol}(\Omega)$. Moreover, "$=$" attains in (\ref{isoineq-in1}) if and only if $\Sigma=S(r^*)$.
\end{theorem} 
\begin{remark}
    It's worth to mention that there are many manifolds which satisfy all of the conditions besides the warped product space with a 1-dimensional factor. For example, we consider that $(M,\bar{g})=(B^{n+1},e^{2\omega}g_{0})$ is the unit disk in $\mathbb{R}^{n+1}$ endowed with a metric $\bar{g}$ which is conformal to the Euclidean metric $g_0=dr^2+r^2g_{S^n}(\theta)$, the conformal factor $\omega$ is a smooth function on $B^{n+1}$. Then, the position vector field $\xi=r\partial_r$ is a conformal vector field, but not closed unless $\omega$ is a radial function. At this time, the leaves of $\mathcal{F}(\xi)$ are the slices $\{r\}\times S^n$ for $r\in(0,1)$. If we choose the conformal factor as $\omega=-\ln(f(r)-rq(\theta))$, where $f(r)$ is a radial function and  $q(\theta)$ is a function on $S^n$. Then, it's easy to check that any slice $\{r\}\times S^n$ has constant mean curvature. And it's not difficult to adjust $f(r)$ and $q(\theta)$ to meet all of the conditions in Theorem \ref{isoineq1}.
\end{remark}

If we don't assume that the inequality (\ref{riccicond-0}) is strict, we need an extra condition on the sectional curvature of the  direction $\xi$ to prove $\Sigma_{\infty}$ is a leaf of the foliation $\mathcal{F}(\xi)$  (see \S \ref{sol to isoper}). Then, the following isoperimetric inequality holds:
\begin{theorem}\label{isoineq2}
    Assume that $(M^{n+1},\bar{g})$ is a Riemannian manifold satisfies all the conditions in Theorem \ref{convergence-0}. Moreover, assume that  the sectional curvature of the direction  $\xi$  satisfies 
	\begin{align}\label{seccurcon}
	K(X,\xi)\geq -\frac{\varphi^2}{\vert\xi\vert^2},\ \forall X\in TM,\ X\neq\xi,
\end{align}
where $K(X,\xi)=\frac{\bar{R}(X,\xi,X,\xi)}{\vert X\vert^2\vert\xi\vert^2-\langle X,\xi\rangle^2}$ is the sectional curvature of $\bar{g}$. Let $\Sigma\subset M'$ be a star-shaped, closed hypersurface which satisfies $\min_{\Sigma}\frac{\vert\xi\vert}{\varphi}>r_1$, then 
	\begin{align*}
		Area(\Sigma)\geq Area(S(r^*)),
	\end{align*}
	where  $r^*$ is the unique real number determined by $\text{Vol}(B(r^*))=\text{Vol}(\Omega)$.
\end{theorem}
\begin{remark}
	The condition (\ref{seccurcon})  implies that the Ricci curvatures of the leaf of $\mathcal{F}(\xi)$ are nonnegative (see \S \ref{sol to isoper} for details).
\end{remark}
Finally, in \S \ref{closed}, we will work on the special case that $\xi$ is closed. At this time, the assumption on the sectional curvature in Theorem \ref{isoineq2} holds if we assume that $\varphi^2-\xi(\varphi)\geq0$, then the isoperimetric inequality holds as a conclusion. But we shall give a new proof of the all time existence and convergence of the solution, and we can release the condition $\varphi^2-\xi(\varphi)>0$ to $\varphi^2-\xi(\varphi)\geq 0$ in the process (see Theorem \ref{conergence-closed}). When $\xi$ is closed, it has been proven that the function $\varphi$ is constant along any leaf of the foliation $\mathcal{F}(\xi)$ (see \cite{Montiel99}). By the identity $\overline{\nabla}\vert\xi\vert^2=2\varphi\xi$ (see Prop \ref{cconformal mfld1}), the leaf of $\mathcal{F}(\xi)$ can be described as the level set of $\vert\xi\vert$ at this time. Let $\bar{S}(r)=\{p\in M|\vert\xi\vert=r\}$ be the level set of the function $\vert\xi\vert$. Given any $r_1\geq 0$, assume that $\Sigma\subset M'$ is a closed, star-shaped hypersurface satisfies $\min_{\Sigma}\vert\xi\vert>r_1$, denote $\Omega$ the bounded domain enclosed by $\Sigma$ and $\bar{S}(r_1)$.  For any $r>r_1$, denote $\bar{B}(r)$ the bounded domain enclosed by $\bar{S}(r)$ and $\bar{S}(r_1)$. 

The isoperimetric inequality reads as follow:

\begin{theorem}\label{isoineq3}
	Assume that $(M^{n+1},\bar{g})$ is a Riemannian manifold endowed with a closed, complete conformal vector field $\xi$ satisfies $\varphi>0$, $\varphi^2-\xi(\varphi)\geq 0$. 
 And we also assume that the condition (iv) in Theorem \ref{convergence-0} holds. Let $\Sigma\subset M'$ be a star-shaped, closed hypersurface which satisfies $\min_{\Sigma}\frac{\vert\xi\vert}{\varphi}>r_1$, then 
	\begin{align*}
		Area(\Sigma)\geq Area(\bar{S}(r^*)),
	\end{align*}
	where $r^*$ is the unique real number determined by $\text{Vol}(\bar{B}(r^*))=\text{Vol}(\Omega)$.
\end{theorem} 
\begin{remark}
	As we known, the warped product space $(M,\bar{g})=(\mathbb{R}\times_{\phi}\mathbf{B}^n,dr^2+\phi^2\tilde{g})$ studied by Guan, Li and Wang \cite{Guan-Li-Wang2019} is a special class of  Riemannian manifolds admitting closed, non-trivial conformal vector fields. In their paper, the conditions $0\leq(\phi')^2-\phi''\phi\leq K$ and $\tilde{Ric}\geq (n-1)K\tilde{g}$ for some constant $K>0$ are needed, where $\tilde{Ric}$ is the Ricci curvature of $\tilde{g}$. Compared with our assumptions, their conditions are stronger than $\varphi^2-\xi(\varphi)\geq 0$ and the condition (iv) on Ricci curvature. 
\end{remark}

We want to state the following conjecture as the end.

\textbf{Conjecture} Assume that a $(n+1)$-dimensional Riemannian manifold $M$ endowed with a codimension one complete foliation $\mathcal{F}$ such that, each leaf of $\mathcal{F}$ is a closed, totally umbilical hypersurface with constant mean curvature. Then, these leaves are the solution of the isoperimetric problem in $M$. 

The rest of this paper is organized as follows: In \S \ref{preliminaries}, we collect and prove some properties in the Riemannian manifold $M$ admitting a non-trivial conformal vector field, especially the Minkowski identities on the hypersurface of $M$. In \S \ref{flow intro}, we introduce the volume preserving flow (\ref{flow-0}) in $M$ and derive some basic evolution equations. In  \S \ref{convergence thm}, we prove the crucial a prior estimates on the mean curvature and the support function, and we complete the proof of Theorem \ref{convergence-0} and Theorem \ref{isoineq1}. In \S \ref{sol to isoper}, we prove the solution hypersurfaces converge to a leaf of the foliation $\mathcal{F}(\xi)$ under the assumption on the sectional curvature and obtain the isoperimetric inequality in Theorem \ref{isoineq2}. In \S \ref{closed}, we work on the case that $\xi$ is closed and prove Theorem \ref{isoineq3}.

\begin{ack}
	The research was surpported by National Key R and D Program of China 2021YFA1001800 and 2020YFA0713100, NSFC 11721101, China Postdoctoral Science Foundation No.2022M723057 and the Fundamental Research Funds for the Central Universities.
\end{ack}

\section{Riemannian manifolds admitting non-trivial conformal vector fields}\label{preliminaries}
\subsection{Preliminaries}
Let $(M^{n+1},\bar{g})$ be a Riemannian manifold with a non-trivial conformal vector field $\xi$, that means
\begin{align}\label{conformal vc1}
	\textit{L}_\xi\bar{g}=2\varphi \bar{g},
\end{align}
where $\varphi=\frac{div\xi}{n+1}$ is the divergence of $\xi$. Then for any vector fields $X,Y\in TM$, we have 
\begin{align}\label{conformalvec}
	\bar{g}(\overline{\nabla}_X\xi,Y)+\bar{g}(\overline{\nabla}_Y\xi,X)=2\varphi\bar{g}(X,Y),
\end{align}
where $\overline{\nabla}$ is the Levi-Civita connection with respect to the metric $\bar{g}$ on $M$.

It's easy to check that the conformal vector field $\xi$ remains  to be conformal under any conformal change of metric on $M$.
Actually, let $\Phi_t$ be a 1-parameter group of conformal transformations of $M$, the vector field on $M$ induced by $\Phi_t$ is a conformal vector field.  Throughout this paper, we assume that $\xi$ is a complete conformal vector field on $M$, that means $\xi$ is induced by a global 1-parameter group. A zero point of $\xi$ is called a singular point which characterizes the fixed point of $\Phi_t$.  Obata \cite[Lemma 2.1]{obata70} showed the following lemma to describe the singular points:
\begin{lemma}
	Let $\xi$ be a complete conformal vector field on $M^{n+1}$, $n\geq 1$. If $\xi$ has a singular point $p_0$ at which its divergence $\varphi$ does not vanish, then $p_0$ is an isolated singular point.
\end{lemma}
Under the same assumption, he classified the Riemmanian manifold admitting a complete conformal vector field with singular points as follow:
\begin{theorem}{\cite{obata70}}
	If a Riemmanian manifolds $M^{n+1}(n\geq 1)$ admitting a complete conformal vector field $\xi$ with singular points at each of which its divergence does not finish, then $M$ is conformally diffeomorphic to either a Euclidean sphere or a Euclidean space.
\end{theorem}

For this reason, we focus on the Riemannian manifold $M$ admitting a complete conformal vector field $\xi$ with $\varphi>0$ in this paper. Then, the set $\mathcal{Z}(\xi)$ consisting of singular points is either a discrete set or a empty set, and we denote the regular set by $M'=M\setminus\mathcal{Z}(\xi)$ which is a open dense set. Next, in order to understand this type of Riemannian manifolds especially their hypersurfaces, we will give some propositions.
\begin{proposition}\label{umbilical}
	 Assume that the n-dimensional distribution $\mathcal{D}$ defined on $M'$ by 
	\[p\in M'\rightarrow\mathcal{D}(p)=\{v\in T_pM|\langle\xi(p),v\rangle=0\}\]
	is integrable and determines a  Riemannian foliation $\mathcal{F}(\xi)$ whose connected leaf is a closed hypersurface oriented by the unit vector field $\mathcal{N}=\frac{\xi}{\vert\xi\vert}$, then each connect leaf of $\mathcal{F}(\xi)$ is  totally umbilical and has mean curvature $H=n\frac{\varphi}{\vert\xi\vert}$.
\end{proposition}
\begin{proof}
	First, the unitary vector field $\mathcal{N}=\frac{\xi}{\vert\xi\vert}$ is well defined on the open set $M'$.  Assume that the closed hypersurface $P$ is a connect leaf of the foliation $\mathcal{F}(\xi)$ on $M'$, and $P$ is oriented by $\mathcal{N}$. Around any point $p\in P$, we can choose a local normal coordinate $\{\tilde{e}_i\}_{i=1}^n$ of $P$, then $\{\tilde{e}_i,\mathcal{N}\}$ spans the tangent vector space $TM$ around $p$. It's easy to compute the second fundamental form of $P$ around the point $p$,
	\begin{align*}
		h^{P}_{ij}=\bar{g}(\overline{\nabla}_{\tilde{e}_i}\mathcal{N},\tilde{e}_j)=\frac{1}{\vert\xi\vert}\bar{g}(\overline{\nabla}_{\tilde{e}_i}\xi,\tilde{e}_j).
	\end{align*}
	The symmetric property of the second fundamental form  immediately leads to
	\begin{align*}
		\bar{g}(\overline{\nabla}_{\tilde{e}_i}\xi,\tilde{e}_j)=\bar{g}(\overline{\nabla}_{\tilde{e}_j}\xi,\tilde{e}_i),
	\end{align*}
	combining this with (\ref{conformalvec}) we have 
	\begin{align}\label{secconformal-eq1}
		\bar{g}(\overline{\nabla}_{\tilde{e}_i}\xi,\tilde{e}_j)=\bar{g}(\overline{\nabla}_{\tilde{e}_j}\xi,\tilde{e}_i)=\varphi\bar{g}_{ij}.
	\end{align}
	Hence, 
	\begin{align}\label{secconformal-eq3}
		h^{P}_{ij}=\frac{\varphi}{\vert\xi\vert}\bar{g}_{ij},
	\end{align}
	which means $P$ is  totally umbilical. Moreover, $P$ has mean curvature 
	\begin{align*}
		H^{P}=n\frac{\varphi}{\vert\xi\vert}.
	\end{align*}
\end{proof}
Next, we assume further that any connect leaf of foliation $\mathcal{F}(\xi)$ has constant mean curvature, then we  can derive some identities in the following proposition. 
 \begin{proposition}\label{conformal vec prop1}
 	Assume that any connect leaf of the totally umbilical Riemannian foliation $\mathcal{F}(\xi)$ on $M'$ is a closed hypersurface with  constant mean curvature. Then  the following formulas hold for any vector field $X\in TM$ on $M'$:
 	\begin{align}
 		\overline{\nabla}\frac{\vert\xi\vert^2}{\varphi^2}=& 2\frac{\varphi^2-\xi(\varphi)}{\varphi^3}\xi  ,\label{secconformal-eq6}\\
 			\overline{\nabla}\vert\xi\vert^2=& 2\frac{\varphi^2-\xi(\varphi)}{\varphi}\xi+2\frac{\vert\xi\vert^2}{\varphi}\overline{\nabla}\varphi,     \label{gradofxi-1}\\
 		\overline{\nabla}_X\xi=& \varphi X+\frac{X(\varphi)}{\varphi}\xi-\frac{\langle X,\xi\rangle}{\varphi}\overline{\nabla}\varphi.\label{1stder} 	
 	\end{align}
 Moreover, we have 
 \begin{align}
 	\overline{\nabla}_X\frac{\xi}{\varphi}=&X-\frac{\langle X,\xi\rangle}{\varphi^2}\overline{\nabla}\varphi  , \label{secconformal-eq5}\\
 	\vert\xi\vert^2\overline{\nabla}\frac{\varphi^2-\xi(\varphi)}{\varphi}=&\xi(\frac{\varphi^2-\xi(\varphi)}{\varphi})\xi     .  \label{secconformal-eq7}
 \end{align}
 \end{proposition}
\begin{proof}
	Since each connect leaf of $\mathcal{F}(\xi)$ has constant mean curvature $n\frac{\varphi}{\vert\xi\vert}$, this implies that 
	\begin{align*}
		\overline{\nabla}\frac{\vert\xi\vert^2}{\varphi^2}=\frac{1}{\vert\xi\vert^2}\xi(\frac{\vert\xi\vert^2}{\varphi^2})\xi=2\frac{\varphi^2-\xi(\varphi)}{\varphi^3}\xi,
	\end{align*}
	where in the second equality we used (\ref{conformalvec}). Then, (\ref{gradofxi-1}) follows from (\ref{secconformal-eq6}).
	
	Next, using (\ref{gradofxi-1}), for any vector field $X\in TM$ we have
	\begin{align}\label{secconformal-eq2}
		\langle\overline{\nabla}_X\xi,\xi\rangle=\frac{\varphi^2-\xi(\varphi)}{\varphi}\langle\xi,X\rangle+\frac{\vert\xi\vert^2}{\varphi}X(\varphi).
	\end{align}
   Then, for any vector field $Z\in TM$ which is perpendicular to $\xi$, we can calculate
	\begin{align}\label{secconformal-eq10}
		\langle\overline{\nabla}_X\xi,Z\rangle=&\frac{\langle\xi,X\rangle}{\vert\xi\vert^2}\langle\overline{\nabla}_{\xi}\xi,Z\rangle+\langle\overline{\nabla}_{X-\frac{\langle\xi,X\rangle}{\vert\xi\vert^2}\xi}\xi,Z\rangle\nonumber\\
		=&-\frac{\langle\xi,X\rangle}{\vert\xi\vert^2}\langle\overline{\nabla}_{Z}\xi,\xi\rangle+\varphi\langle X,Z\rangle\nonumber\\
		=&\varphi\langle X,Z\rangle-\frac{\langle\xi,X\rangle}{\varphi}Z(\varphi),
	\end{align}
where in the second equality we used (\ref{conformalvec}) and (\ref{secconformal-eq1}), in the last equality we used (\ref{secconformal-eq2}).
Combining  (\ref{secconformal-eq2})  with  (\ref{secconformal-eq10}), we immediately obtain (\ref{1stder}) and (\ref{secconformal-eq5}) as a conclusion.

Finally, combining (\ref{secconformal-eq5}) with (\ref{secconformal-eq6}), for any vector fields $X,Y\in TM$ we have 
\begin{align*}
	\langle\overline{\nabla}_{X}\overline{\nabla}\frac{\vert\xi\vert^2}{\varphi^2},Y\rangle=&	2\langle\overline{\nabla}_{X}\left(\frac{\varphi^2-\xi(\varphi)}{\varphi^3}\xi\right),Y\rangle\\
	=&\frac{2}{\varphi^2}X(\frac{\varphi^2-\xi(\varphi)}{\varphi})\langle\xi,Y\rangle-2\frac{\varphi^2-\xi(\varphi)}{\varphi^4}X(\varphi)\langle\xi,Y\rangle\\
	&+2\frac{\varphi^2-\xi(\varphi)}{\varphi^2}\langle X,Y\rangle-2\frac{\varphi^2-\xi(\varphi)}{\varphi^4}Y(\varphi)\langle\xi,X\rangle.
\end{align*}
Then, by symmetric property of the Hessian operator, we see that
\begin{align*}
	0=&\langle\overline{\nabla}_{X}\overline{\nabla}\frac{\vert\xi\vert^2}{\varphi^2},Y\rangle-\langle\overline{\nabla}_{Y}\overline{\nabla}\frac{\vert\xi\vert^2}{\varphi^2},X\rangle\\
	=&\frac{2}{\varphi^2}X(\frac{\varphi^2-\xi(\varphi)}{\varphi})\langle\xi,Y\rangle-\frac{2}{\varphi^2}Y(\frac{\varphi^2-\xi(\varphi)}{\varphi})\langle\xi,X\rangle.
\end{align*}
Hence,
\begin{align*}
	X(\frac{\varphi^2-\xi(\varphi)}{\varphi})\langle\xi,Y\rangle=Y(\frac{\varphi^2-\xi(\varphi)}{\varphi})\langle\xi,X\rangle,
\end{align*}
by choosing $Y=\xi$ we get (\ref{secconformal-eq7}).
\end{proof}
\subsection{Minkowski identities}
Assume that $\Sigma$ is a hypersurface embedded in $M'$, we want to derive some Minkowski identities on $\Sigma$  which will play crucial roles in the definition of the flow (\ref{flow-0}). 

First, we need to compute the sectional curvatures and Ricci curvatures of the direction $\xi$. 
\begin{proposition}\label{conformal vec prop2}
	Under the same conditions as that in Proposition \ref{conformal vec prop1}, for any vector field $X\in TM$, on $M'$ we have
	\begin{align}\label{curvature-eq1}
		\bar{R}(X,\xi,X,\xi)
		=&-\frac{\vert\xi\vert^2}{\varphi}\langle\overline{\nabla}_{X}\overline{\nabla}\varphi,X\rangle+\frac{\langle\xi,X\rangle^2}{\varphi}\langle\overline{\nabla}_{\mathcal{N}}\overline{\nabla}\varphi,\mathcal{N}\rangle
	\end{align}
	and
	\begin{align}\label{curvature-eq2}
		\overline{Ric}(\xi,X)=-\frac{\langle\xi,X\rangle}{\varphi}(\overline{\Delta}\varphi-\langle\overline{\nabla}_{\mathcal{N}}\overline{\nabla}\varphi,\mathcal{N}\rangle),
	\end{align}
where $\overline{\Delta}$ is the Laplace operator with respect to the metric $\bar{g}$ on $M$.

	Moreover, $\forall v\bot\xi$ we have 
	\begin{align}\label{curvature-eq3}
		\overline{Ric}(\xi,v)=0\ \text{and}\ \langle\overline{\nabla}_v\overline{\nabla}\varphi,\xi\rangle=0.
	\end{align}
\end{proposition}
\begin{proof}
	Around any point $p\in M'$, choosing a local normal coordinate $\{\bar{e}_i\}_{i=1}^{i=n+1}$ of $TM$, by the definition of curvature tensor, using (\ref{secconformal-eq5}) we have
	\begin{align*}
		\langle\bar{R}(\bar{e}_i,\bar{e}_k)\frac{\xi}{\varphi},\bar{e}_{\ell}\rangle
		=&\langle\overline{\nabla}_{\bar{e}_i}\overline{\nabla}_{\bar{e}_k}\frac{\xi}{\varphi}-\overline{\nabla}_{\bar{e}_k}\overline{\nabla}_{\bar{e}_i}\frac{\xi}{\varphi},\bar{e}_{\ell}\rangle\\
		=&\langle\overline{\nabla}_{\bar{e}_i}\left(\bar{e}_k-\frac{\langle\xi,\bar{e}_k\rangle}{\varphi^2}\overline{\nabla}\varphi\right),\bar{e}_{\ell}\rangle-\langle\overline{\nabla}_{\bar{e}_k}\left(\bar{e}_i-\frac{\langle\xi,\bar{e}_i\rangle}{\varphi^2}\overline{\nabla}\varphi\right),\bar{e}_{\ell}\rangle\\
		=&\left(\bar{e}_k\left(\frac{\langle\xi,\bar{e}_i\rangle}{\varphi^2}\right)-\bar{e}_i\left(\frac{\langle\xi,\bar{e}_k\rangle}{\varphi^2}\right)\right)\bar{e}_{\ell}(\varphi)-\frac{\langle\xi,\bar{e}_k\rangle}{\varphi^2}\langle\overline{\nabla}_{\bar{e}_i}\overline{\nabla}\varphi,\bar{e}_{\ell}\rangle\\
  &+\frac{\langle\xi,\bar{e}_i\rangle}{\varphi^2}\langle\overline{\nabla}_{\bar{e}_k}\overline{\nabla}\varphi,\bar{e}_{\ell}\rangle\\
		=&\frac{\langle\overline{\nabla}_{\bar{e}_k}\xi,\bar{e}_i\rangle-\langle\overline{\nabla}_{\bar{e}_i}\xi,\bar{e}_k\rangle}{\varphi^2}\bar{e}_{\ell}(\varphi)+2\frac{\langle\xi,\bar{e}_k\rangle}{\varphi^3}\bar{e}_i(\varphi)\bar{e}_{\ell}(\varphi)\\
		&-2\frac{\langle\xi,\bar{e}_i\rangle}{\varphi^3}\bar{e}_{k}(\varphi)\bar{e}_{\ell}(\varphi)-\frac{\langle\xi,\bar{e}_k\rangle}{\varphi^2}\langle\overline{\nabla}_{\bar{e}_i}\overline{\nabla}\varphi,\bar{e}_{\ell}\rangle\\
  &+\frac{\langle\xi,\bar{e}_i\rangle}{\varphi^2}\langle\overline{\nabla}_{\bar{e}_k}\overline{\nabla}\varphi,\bar{e}_{\ell}\rangle.
	\end{align*}
By (\ref{1stder}), 
\begin{align*}
	\langle\overline{\nabla}_{\bar{e}_k}\xi,\bar{e}_i\rangle-\langle\overline{\nabla}_{\bar{e}_i}\xi,\bar{e}_k\rangle=2\frac{\langle\xi,\bar{e}_i\rangle}{\varphi}\bar{e}_k(\varphi)-2\frac{\langle\xi,\bar{e}_k\rangle}{\varphi}\bar{e}_i(\varphi).
\end{align*}
Hence,
\begin{align}\label{eq-11}
		\langle\bar{R}(\bar{e}_i,\bar{e}_k)\frac{\xi}{\varphi},\bar{e}_{\ell}\rangle
	=&-\frac{\langle\xi,\bar{e}_k\rangle}{\varphi^2}\langle\overline{\nabla}_{\bar{e}_i}\overline{\nabla}\varphi,\bar{e}_{\ell}\rangle+\frac{\langle\xi,\bar{e}_i\rangle}{\varphi^2}\langle\overline{\nabla}_{\bar{e}_k}\overline{\nabla}\varphi,\bar{e}_{\ell}\rangle.
\end{align}
	Then,
	\begin{align}\label{eq-10}
		\overline{Ric}(\xi,\cdot)=&\bar{g}^{i\ell}\langle\bar{R}(\bar{e}_i,\cdot)\xi,\bar{e}_{\ell}\rangle=-\frac{\overline{\Delta}\varphi}{\varphi}\langle\xi,\cdot\rangle+\frac{1}{\varphi}\langle\overline{\nabla}_{\cdot}\overline{\nabla}\varphi,\xi\rangle.
	\end{align}
On the other hand, assume that $P$ is the connect leaf of the foliation $\mathcal{F}(\xi)$ such that $p\in P$. Taking a local normal coordinate $\{\tilde{e}_i\}_{i=1}^{n}$ of $P$, since $P$ is totally umbilical and has constant mean curvature, by Codazzi equation we have
\begin{align*}
	\overline{R}(\mathcal{N},\tilde{e}_k, \tilde{e}_i,\tilde{e}_j)=h^{P}_{ki,j}-h^{P}_{kj,i}=0
\end{align*}
Therefore,
\begin{align*}
    \overline{Ric}(\mathcal{N}, \tilde{e}_i)=0,\ \forall i=1,\cdots,n.
\end{align*}
Substituting it into (\ref{eq-10}), we have
	\begin{align*}
		\langle\overline{\nabla}_{\tilde{e}_i}\overline{\nabla}\varphi,\xi\rangle=0,\ \forall i=1,\cdots,n,
	\end{align*}
this is (\ref{curvature-eq3}).
Now, (\ref{curvature-eq2}) follows from (\ref{curvature-eq3}) and (\ref{eq-10}) immediately.	
In addition, we get (\ref{curvature-eq1}) by plugging (\ref{curvature-eq3}) into (\ref{eq-11}).
\end{proof}

Next, we will prove Minkowski identities on a closed hypersurface $\Sigma\subset M'$. Using local frames $\{e_i\}_{i=1}^n$ on $\Sigma$, denote by $g_{ij}$, $h_{ij}$ and $h^i_j=g^{ik}h_{kj}$   the induced metric, the second fundamental form and the Weingarten tensor respectively for $i,j=1,\cdots,n$. We also let $\sigma_{\ell}$ denote the $\ell$-th elementary symmetric functions of the principle curvatures of $\Sigma$. First, we recall the following formula which holds in any Riemannian manifold (see \cite{Guan-Li-Wang2019} for the detail).
\begin{lemma}
	Let $\sigma_2^{ij}=Hg^{ij}-h^{ij}$ be the cofactor tensor. Then the trace of its covariant derivative is 
	\begin{align}\label{sigmadiv}
		\sigma_2^{ij}(h)_j=-\overline{Ric}(e_i,\nu),
	\end{align}
	where $\nu$ is the unit outward normal of the hypersurface.
\end{lemma} 
Then, we obtain the following Minkowski identities.
\begin{lemma}\label{lemma-minikowsi}
	Let $(M^{n+1},\bar{g})$ be a Riemannian manifold endowed with a complete conformal vector field $\xi$. Then on the hypersurface $\Sigma\subset M'$ we have,
	\begin{align}\label{minikowski0}
		\int_{\Sigma}n\varphi-uHd\mu=0,
	\end{align}
and 
	\begin{align}\label{minikowski}
		(n-1)\int_{\Sigma}\varphi\sigma_1d\mu=2\int_{\Sigma}\sigma_2ud\mu+\int_{\Sigma}\overline{Ric}(\xi,\nu)-u\overline{Ric}(\nu,\nu)d\mu,
	\end{align}
	where $\nu$ is the unit normal vector field of $\Sigma$ and $u=\langle\xi,\nu\rangle$ is the support function. 
	
	If we assume further that  any connect leaf of the foliation $\mathcal{F}(\xi)$ has constant mean curvature, then the second Minkowski identity (\ref{minikowski}) implies that 
 \begin{align}\label{minikowski2}
     (n-1)\int_{\Sigma}\varphi\sigma_1d\mu=2\int_{\Sigma}\sigma_2ud\mu+\int_{\Sigma}u\left(\overline{Ric}(\mathcal{N},\mathcal{N})-\overline{Ric}(\nu,\nu)\right)d\mu.
 \end{align}
\end{lemma}
\begin{proof}
  If $\Sigma$ is a leaf of the foliation $\mathcal{F}(\xi)$, then $u=\vert\xi\vert$. Since $\Sigma$ is totally umbilical and $H=n\frac{\varphi}{\vert\xi\vert}$, the identities in lemma holds evidently.
  
	Otherwise, we denote the tangential part of the vector field $\xi$ by $\xi^{\top}$ on $\Sigma$. Using (\ref{conformal vc1}), we have
	\begin{align*}
		div_{\Sigma}(\xi^{\top})&=\langle\nabla_i(\xi-u\nu),e_i\rangle=\langle\overline{\nabla}_i\xi,e_i\rangle-uH\\
		&=\frac{1}{2}\textit{L}_\xi\bar{g}(e_i,e_i)-uH=n\varphi-uH.
	\end{align*}
	Integrating the above identity, (\ref{minikowski0}) holds by Stokes' formula.
	
	Next, denote the vector field $W$ on $\Sigma$ by
	\[W:=\sigma_2^{ij}\langle\xi,e_j\rangle e_i,\]
	then combining (\ref{conformal vc1}) with (\ref{sigmadiv}) we get
	\begin{align*}
		div_{\Sigma}W&=div_{\Sigma}(\sigma_2^{ij}\langle\xi,e_j\rangle e_i)\\
		&=\sigma_2^{ij}(h)_i\langle\xi,e_j\rangle+\sigma_2^{ij}e_i\langle\xi,e_j\rangle\\
		&=-\overline{Ric}(e_j,\nu)\langle\xi,e_j\rangle+\varphi\sigma_2^{ij} g_{ij}-\sigma_2^{ij}h_{ij}\langle\xi,\nu\rangle,
	\end{align*}
 where in the second identity we used (\ref{conformal vc1}) and the symmetric property of $\sigma_2^{ij}$.
 
	Using  $\sigma_2^{ij}=Hg^{ij}-h^{ij}$ and $\xi=\langle\xi,e_j\rangle e_j+u\nu$, we derive that
	\begin{align*}
		div_{\Sigma}W&=-\overline{Ric}(\xi,\nu)+u\overline{Ric}(\nu,\nu)+\varphi(Hg^{ij}-h^{ij})g_{ij} -u(Hg^{ij}-h^{ij})h_{ij}\\
		&=-\overline{Ric}(\xi,\nu)+u\overline{Ric}(\nu,\nu)+(n-1)\varphi H-2\sigma_2u,
	\end{align*}
 Now, integrating the above identity leads to (\ref{minikowski}). Finally, if any connect leaf of the foliation $\mathcal{F}(\xi)$ has constant mean curvature, then (\ref{curvature-eq3}) implies that 
 \begin{align*}
    \overline{Ric}(\xi,\nu)=u\overline{Ric}(\mathcal{N},\mathcal{N}),
 \end{align*}
 so (\ref{minikowski2}) follows.
\end{proof}

\section{The volume preserving flow}\label{flow intro}
In this section, we will introduce the mean curvature type flow which is the one introduced  by P. Guan and J. Li \cite{guan-li15} in space forms. The ambient space we are interested  in  is the Riemannian manifold endowed with a complete conformal vector field satisfies conditions (i), (ii), (iii) and (iv) in Theorem \ref{convergence-0}.
 
Let  $\Sigma(t)\subset M^{n+1}$ be a smooth family of closed hypersurfaces, $F:\Sigma\times[0,T)\rightarrow M$ be the embedded map such that $F(\Sigma,t)=\Sigma(t)$. We consider the following flow in $M$:
\begin{equation}\label{flow}
	\left\{
	\begin{aligned}
		&&\frac{\partial F}{\partial t}=(n\varphi-uH)\nu,\\
		&&F(\cdot,0)=\Sigma(0),
	\end{aligned}
	\right.
\end{equation}
where $\nu$ is the unit outward normal of $\Sigma(t)$, $u=\langle\xi,\nu\rangle$ is the support function.
\begin{theorem}\label{monotonethm}
	Assume that $(M^{n+1},\bar{g})$ is a Riemannian manifold endowed with a complete conformal vector field which satisfies the  conditions (i), (iii) and (iv) in Theorem \ref{convergence-0}.
	Let $\Sigma(t)\subset M'$ be a smooth closed solution of  the flow (\ref{flow}) on the time interval $[0,T)$ and $\Omega(t)$ be  the bounded domains enclosed by $\Sigma(t)$ and a fixed closed hypersurface. We also assume that $\Sigma(t)$ are star-shaped hypersurfaces, i.e $u=\langle\xi,\nu\rangle> 0$ on $\Sigma(t)$. 
Then the volume of $\Omega(t)$ is invariant and the  area of $\Sigma(t)$ is non-increasing along the flow.
\end{theorem}
\begin{proof}
Denote the volume of $\Omega(t)$ by $V(t)$ and the area of $\Sigma(t)$ by $A(t)$, along the flow (\ref{flow}) we have
	\begin{align*}
		V'(t)=\int_{\Sigma(t)}(n\varphi-uH)d\mu,
	\end{align*}
\begin{align*}
	A'(t)=\int_{\Sigma(t)}(n\varphi-uH)Hd\mu.
\end{align*}
By the first Minkowski identity (\ref{minikowski0}), $V'(t)=0$, this means the enclosed volume is constant along the flow.

Next, we prove that the area of $\Sigma(t)$ is non-increasing along the flow (\ref{flow}). Using the Minkowski identity (\ref{minikowski2}), we have
\begin{align}\label{mono-eq1}
	A'(t)&=\int_{\Sigma(t)}(n\varphi-uH)Hd\mu\nonumber\\
	&=\int_{\Sigma}\left(\frac{2n}{n-1}\sigma_2-H^2+\frac{n}{n-1}(\overline{Ric}(\mathcal{N},\mathcal{N})-\overline{Ric}(\nu,\nu))\right)ud\mu\nonumber\\
	&\leq \frac{n}{n-1}\int_{\Sigma}(\overline{Ric}(\mathcal{N},\mathcal{N})-\overline{Ric}(\nu,\nu))ud\mu,
\end{align}
where in the last inequality we use  the Newton-Maclaurin inequality $H^2\geq\frac{2n}{n-1}\sigma_2$ and the assumption $u\geq 0$.
Then by the condition (iv) in Theorem \ref{convergence-0}, we have
\begin{align*}
	A'(t)&\leq 0.
\end{align*}
\end{proof}

Next, we summarize the evolution equations in the following proposition.
\begin{proposition}[Evolution equations]
Assume that $(M^{n+1},\bar{g})$ is a Riemannian manifold endowed with a complete conformal vector field which satisfies the conditions (i) and (iii) in Theorem \ref{convergence-0}.
Let $\Sigma(t)\subset M'$ be a smooth closed solution of  the flow (\ref{flow}) on the time interval $[0,T)$.  Then under the flow (\ref{flow}), we derive the following evolution equations:
	\begin{align}
		\partial_tg_{ij}&=2(n\varphi-uH)h_{ij}, \label{metric1}\\
		\partial_t\nu&=-\nabla(n\varphi-uH), \label{norm vec}\\
		\partial_th_{ij}&=-\nabla_{\partial_iF}\nabla_{\partial_jF}(n\varphi-uH)+(n\varphi-uH)h_{ik}h^k_j\nonumber\\
  &\quad-(n\varphi-uH)\bar{R}_{i\nu j\nu}  ,\label{second form1}
	\end{align}
\end{proposition}

\begin{proof}
 The proof of (\ref{metric1}) and (\ref{norm vec}) is standard (see e.g \cite{huisken1984flow,guan-li15,Guan-Li-Wang2019}).

	By definition we have
	\begin{align*}
		h_{ij}=-\langle\overline{\nabla}_{\partial_iF}\partial_jF,\nu\rangle.
	\end{align*}
	Along the flow (\ref{flow}), we compute
	\begin{align}
		\partial_th_{ij}&=-\partial_t\langle\overline{\nabla}_{\partial_iF}\partial_jF,\nu\rangle\nonumber\\
		&=-\langle\overline{\nabla}_{\partial_tF}\overline{\nabla}_{\partial_iF}\partial_jF,\nu\rangle-\langle\overline{\nabla}_{\partial_iF}\partial_jF,\partial_t\nu\rangle\nonumber\\
		&=-\langle\overline{\nabla}_{\partial_iF}\overline{\nabla}_{\partial_tF}\partial_jF,\nu\rangle+\langle\bar{R}(\partial_iF,\partial_tF)\partial_jF,\nu\rangle
  +h_{ij}\langle\nu,\partial_t\nu\rangle\nonumber\\
		&=-\langle\overline{\nabla}_{\partial_iF}\overline{\nabla}_{\partial_jF}((n\varphi-uH)\nu),\nu\rangle-(n\varphi-uH)\bar{R}_{i\nu j\nu}\nonumber\\
		&=-\overline{\nabla}_{\partial_iF}\overline{\nabla}_{\partial_jF}(n\varphi-uH)-(n\varphi-uH)\langle\overline{\nabla}_{\partial_iF}\overline{\nabla}_{\partial_jF}\nu,\nu\rangle\nonumber\\
  &\quad-(n\varphi-uH)\bar{R}_{i\nu j\nu}.\label{eq-5}
	\end{align}
$\langle\overline{\nabla}_i\nu,\nu\rangle=0$ implies 
\begin{align}\label{eq-6}
	\langle\overline{\nabla}_i\overline{\nabla}_j\nu,\nu\rangle=-\langle\overline{\nabla}_i\nu,\overline{\nabla}_j\nu\rangle=-h_{i\ell}h^{\ell}_j.
\end{align}
	Combining (\ref{eq-5}) with (\ref{eq-6}), we have (\ref{second form1}).
\end{proof}

We now show that the function $\frac{\vert\xi\vert^2}{\varphi^2}$ is uniform bounded from above and below by the initial data.
\begin{proposition}\label{C0est-2}
	Assume that $(M^{n+1},\bar{g})$ is a Riemannian manifold endowed with a complete conformal vector field satisfying the conditions (i), (ii) and (iii) in Theorem \ref{convergence-0}.  Let $\Sigma_0$ be a closed hypersurface embedded in $M'$, if $\Sigma(t)$ is a smooth solution of the flow (\ref{flow}) starting from $\Sigma_0$.
	Then for any $(p,t)\in\Sigma\times[0,T)$ we have
	\begin{align}\label{C0est2}
		\min_{p\in\Sigma}\frac{\vert\xi\vert^2}{\varphi^2}(p,0)\leq \frac{\vert\xi\vert^2}{\varphi^2}(p,t)\leq \max_{p\in\Sigma}\frac{\vert\xi\vert^2}{\varphi^2}(p,0). 
	\end{align}
	From this, we conclude that $\Sigma(t)\cap\mathcal{Z}(\xi)=\emptyset$ along the flow.
\end{proposition}
\begin{remark}\label{C0remark}
	\item 
	\begin{itemize}
		\item[(1)]	The assumption $\varphi^2-\xi(\varphi)>0$ is necessary for the technique. Recall that the mean curvature of any leaf of the foliation $\mathcal{F}(\xi)$ is $n\frac{\varphi}{\vert\xi\vert}$. By direct computation we have 	
		\begin{align*}
			\xi(\frac{\varphi}{\vert\xi\vert})=-\frac{\varphi^2-\xi(\varphi)}{\vert\xi\vert}.
		\end{align*}
		Therefore, the condition $\varphi^2-\xi(\varphi)>0$ implies that the mean curvatures of leaves of the foliation $\mathcal{F}(\xi)$ are strictly decreasing along the direction $\xi$. 
		\item[(2)] The bounds of $\frac{\vert\xi\vert^2}{\varphi^2}$ implies that the evolving hypersurfaces are restricted in the compact domain enclosed by two leaves of the foliation $\mathcal{F}(\xi)$ determined by $\Sigma_0$.
	\end{itemize}
\end{remark}
\begin{proof}
	Along the flow (\ref{flow}), by (\ref{secconformal-eq6}) we have 
	\begin{align*}
		\partial_t\frac{\vert\xi\vert^2}{\varphi^2}=(n\varphi-uH)\nu(\frac{\vert\xi\vert^2}{\varphi^2})=2\frac{\varphi^2-\xi(\varphi)}{\varphi^3}(n\varphi-uH)u.
	\end{align*}
	Using (\ref{secconformal-eq6}) again,
	\begin{align}\label{C0est-eq2}
		\nabla_i\frac{\vert\xi\vert^2}{\varphi^2}=2\frac{\varphi^2-\xi(\varphi)}{\varphi^3}\langle\xi,e_i\rangle.
	\end{align}
Moreover, by (\ref{conformalvec}) and (\ref{secconformal-eq7}) we get
	\begin{align*}
		u\Delta_g\frac{\vert\xi\vert^2}{\varphi^2}=&2e_i\left(\frac{\varphi^2-\xi(\varphi)}{\varphi}\right)\frac{1}{\varphi^2}\langle\xi,e_i\rangle u-4\frac{\varphi^2-\xi(\varphi)}{\varphi^4}e_i(\varphi)\langle\xi,e_i\rangle u\\
		&+2\frac{\varphi^2-\xi(\varphi)}{\varphi^3}\langle\overline{\nabla}_i\xi,e_i\rangle u-2\frac{\varphi^2-\xi(\varphi)}{\varphi^3}u^2H\\
		=&2\frac{u}{\varphi^2\vert\xi\vert^2}\xi\left(\frac{\varphi^2-\xi(\varphi)}{\varphi}\right)(\vert\xi\vert^2-u^2)-4\frac{\varphi^2-\xi(\varphi)}{\varphi^4}e_i(\varphi)\langle\xi,e_i\rangle u\\
		&+2n\frac{\varphi^2-\xi(\varphi)}{\varphi^2}u-2\frac{\varphi^2-\xi(\varphi)}{\varphi^3}u^2H.
	\end{align*}
	Then
	\begin{align}\label{C0est-eq1}
		\partial_t\frac{\vert\xi\vert^2}{\varphi^2}=&u\Delta_g\frac{\vert\xi\vert^2}{\varphi^2}-2\frac{u}{\varphi^2\vert\xi\vert^2}\xi\left(\frac{\varphi^2-\xi(\varphi)}{\varphi}\right)(\vert\xi\vert^2-u^2)\nonumber\\
  &+4\frac{\varphi^2-\xi(\varphi)}{\varphi^4}e_i(\varphi)\langle\xi,e_i\rangle u.
	\end{align}
	Since $\varphi>0$ and $\varphi^2-\xi(\varphi)>0$, it follows from  (\ref{C0est-eq2}) that at the critical points we have 
	\begin{align*}
		\langle\xi,e_i\rangle=0,\ \forall i=1,\cdots,n.
	\end{align*}
	Together with (\ref{C0est-eq1}), at critical points, 
	\begin{align*}
		\partial_t\frac{\vert\xi\vert^2}{\varphi^2}=&u\Delta_g\frac{\vert\xi\vert^2}{\varphi^2}.
	\end{align*}
	By the standard maximum principle, we get the uniform bounds of $\frac{\vert\xi\vert^2}{\varphi^2}$.
\end{proof}
\section{Convergence theorem of the volume preserving flow}\label{convergence thm}
In this section, under the condition (i), (ii), (iii) and (iv) in Theorem \ref{convergence-0}, we will prove the uniform upper bound for the mean curvature $H$ and the uniform positive lower bound for the support function $u$ along the flow. The proofs of above estimates follow similar ideas in \cite{Guan-Li-Wang2019}. As a conclusion, we can show that the flow (\ref{flow}) exists for all time with the star-shaped initial data $\Sigma_0$. Moreover, the solution hypersurfaces converge smoothly to a totally umbilical hypersurface.

First, we derive the evolution equations of the support function $u$ and the mean curvature $H$.
\begin{proposition}
		Assume that $(M^{n+1},\bar{g})$ is a Riemannian manifold endowed with a complete conformal vector field $\xi$ satisfying the condition (i) and (iii) in Theorem \ref{convergence-0}. Let $\Sigma_0$ be a closed hypersurface embedded in $M'$, then under the flow (\ref{flow}), the support function $u=\langle\xi,\nu\rangle$ evolves by
	\begin{align}\label{ev-support fun}
		\partial_tu=&u\Delta_gu+H\langle\xi,\nabla u\rangle+\vert A\vert^2u^2-2\varphi uH+n\varphi^2-n\xi(\varphi)\nonumber\\
  &+2n\nu(\varphi)u+u^2\overline{Ric}(\nu,\nu).
	\end{align}
\end{proposition}
\begin{proof}
  First, along the flow (\ref{flow}),	by (\ref{conformalvec}) and (\ref{norm vec}) we have 
	\begin{align}\label{suppfun-eq0}
		\partial_tu=&\langle\overline{\nabla}_{\partial_tF}\xi,\nu\rangle+\langle\xi,\partial_t\nu\rangle\nonumber\\
		=&(n\varphi-uH)\langle\overline{\nabla}_{\nu}\xi,\nu\rangle-\langle\xi,\nabla(n\varphi-uH)\rangle\nonumber\\
		=&n\varphi^2-\varphi uH-n\xi(\varphi)+nu\nu(\varphi)+\langle\xi,\nabla(uH)\rangle.
	\end{align}
Next, using (\ref{1stder}), we obtain
	\begin{align}\label{suppfun-eq1}
		\nabla_iu=&\langle\overline{\nabla}_i\xi,\nu\rangle+h_{ik}\langle\xi,e_k\rangle\nonumber\\
		=&u\frac{\langle\overline{\nabla}\varphi,e_i\rangle}{\varphi}-\frac{\langle\xi,e_i\rangle}{\varphi}\langle\overline{\nabla}\varphi,\nu\rangle+h_{ik}\langle\xi,e_k\rangle.
	\end{align}
Hence, using (\ref{1stder}) again, we have 
	\begin{align}\label{suppfun-eq2}
		\Delta_gu=&\frac{\langle\nabla u,\nabla\varphi\rangle}{\varphi}-\frac{u}{\varphi^2}\vert\nabla\varphi\vert^2+\frac{u}{\varphi}\langle\overline{\nabla}_i\overline{\nabla}\varphi,e_i\rangle-n\nu(\varphi)\nonumber\\
		&+\frac{\langle\xi,e_i\rangle e_i(\varphi)}{\varphi^2}\nu(\varphi)-\frac{\langle\xi,e_i\rangle}{\varphi}\langle\overline{\nabla}_i\overline{\nabla}\varphi,\nu\rangle-h_{ik}\frac{\langle\xi,e_i\rangle e_k(\varphi)}{\varphi}\\
		&+h_{ik,i}\langle\xi,e_k\rangle+\varphi H-\vert A\vert^2u.\nonumber
	\end{align}
From (\ref{suppfun-eq1}), we see that
\begin{align}\label{suppfun-eq3}
	&\frac{\langle\nabla u,\nabla\varphi\rangle}{\varphi}=\frac{u}{\varphi^2}\vert\nabla\varphi\vert^2-\frac{\langle\xi,e_i\rangle e_i(\varphi)}{\varphi^2}\nu(\varphi)+h_{ik}\frac{\langle\xi,e_i\rangle e_k(\varphi)}{\varphi}.
\end{align}
 Recall (\ref{curvature-eq3}), we have
\begin{align}\label{suppfun-eq4}
 \langle\xi,e_i\rangle \langle\overline{\nabla}_i\overline{\nabla}\varphi,\nu\rangle=& \langle\overline{\nabla}_{\xi}\overline{\nabla}\varphi,\nu\rangle- \langle\overline{\nabla}_{\nu}\overline{\nabla}\varphi,\nu\rangle u\nonumber\\
 =&\langle\overline{\nabla}_{\mathcal{N}}\overline{\nabla}\varphi,\mathcal{N}\rangle u- \langle\overline{\nabla}_{\nu}\overline{\nabla}\varphi,\nu\rangle u.
\end{align}
Substituting (\ref{suppfun-eq3}) and (\ref{suppfun-eq4}) into (\ref{suppfun-eq2}), it reduces to
\begin{align}\label{suppfun-eq5}
		\Delta_gu
		=&\frac{u}{\varphi}\langle\overline{\nabla}_i\overline{\nabla}\varphi,e_i\rangle-\frac{u}{\varphi}\langle\overline{\nabla}_{\mathcal{N}}\overline{\nabla}\varphi,\mathcal{N}\rangle +\frac{u}{\varphi}\langle\overline{\nabla}_{\nu}\overline{\nabla}\varphi,\nu\rangle-n\nu(\varphi)\nonumber\\
  &+h_{ik,i}\langle\xi,e_k\rangle+\varphi H-\vert A\vert^2u\nonumber\\
		=&\frac{u}{\varphi}(\overline{\Delta}\varphi-\langle\overline{\nabla}_{\mathcal{N}}\overline{\nabla}\varphi,\mathcal{N}\rangle)-n\nu(\varphi)+\varphi H-\vert A\vert^2u+h_{ik,i}\langle\xi,e_k\rangle\nonumber\\
		=&-\overline{Ric}(\xi,\nu)-n\nu(\varphi)+\varphi H-\vert A\vert^2u+h_{ik,i}\langle\xi,e_k\rangle,
\end{align}
where in the last equality we used (\ref{curvature-eq2}). Now, by Codazzi equation 
	\begin{align}\label{suppfun-eq6}
		h_{ik,i}\langle\xi,e_k\rangle=&h_{ii,k}\langle\xi,e_k\rangle+\bar{R}_{\nu iki}\langle\xi,e_k\rangle\nonumber\\
		=&\langle\xi,\nabla
		H\rangle+\overline{Ric}(\xi,\nu)-u\overline{Ric}(\nu,\nu).
	\end{align}
Combining (\ref{suppfun-eq5}) with (\ref{suppfun-eq6}) we have 
	\begin{align}\label{suppfun-eq7}
		\Delta_gu=&\langle\xi,\nabla
		H\rangle-n\nu(\varphi)+\varphi H-\vert A\vert^2u-u\overline{Ric}(\nu,\nu).
	\end{align}
 Plugging this into (\ref{suppfun-eq0}), we obtain (\ref{ev-support fun}).
\end{proof}
\begin{proposition}
		Assume that $(M^{n+1},\bar{g})$ is a Riemannian manifold endowed with a complete conformal vector field satisfying the conditions (i) and (iii) in Theorem \ref{convergence-0}. Let $\Sigma_0$ be a closed hypersurface embedded in $M'$, then under the flow (\ref{flow}), the mean curvature $H$ evolves by
	\begin{align}\label{meancurv}
		\partial_tH=&u\Delta_gH+2\langle\nabla u,\nabla H\rangle+H\langle\xi,\nabla H\rangle+\varphi H^2-n\varphi\vert A\vert^2\nonumber\\
		&+n\varphi\left(\overline{Ric}(\mathcal{N},\mathcal{N})-\overline{Ric}(\nu,\nu)\right)+n\left(\langle\overline{\nabla}_{\nu}\overline{\nabla}\varphi,\nu\rangle-\langle\overline{\nabla}_{\mathcal{N}}\overline{\nabla}\varphi,\mathcal{N}\rangle\right).
	\end{align}
\end{proposition}
\begin{proof}
	By (\ref{metric1}) and (\ref{second form1}) we have
	\begin{align}\label{meancur-eq0}
		\partial_tH=&\partial_t(g^{ij}h_{ij})\nonumber\\
  =&-\Delta_g(n\varphi-uH)-(n\varphi-uH)\vert A\vert^2-(n\varphi-uH)\overline{Ric}(\nu,\nu)\nonumber\\
		=&u\Delta_gH+H\Delta_g u+2\langle\nabla u,\nabla H\rangle-n\Delta_g\varphi-(n\varphi-uH)\vert A\vert^2\nonumber\\
		&-(n\varphi-uH)\overline{Ric}(\nu,\nu). 
	\end{align}
Next, using (\ref{curvature-eq2}) we get
\begin{align}\label{divfun-eq2}
	\Delta_g\varphi=&e_i\langle\overline{\nabla}\varphi,e_i\rangle=\langle\overline{\nabla}_i\overline{\nabla}\varphi,e_i\rangle-H\nu(\varphi)\nonumber\\
	=&\overline{\Delta}\varphi-\langle\overline{\nabla}_{\nu}\overline{\nabla}\varphi,\nu\rangle-H\nu(\varphi)\nonumber\\
	=&-\varphi\overline{Ric}(\mathcal{N},\mathcal{N})+\langle\overline{\nabla}_{\mathcal{N}}\overline{\nabla}\varphi,\mathcal{N}\rangle-\langle\overline{\nabla}_{\nu}\overline{\nabla}\varphi,\nu\rangle-H\nu(\varphi).
\end{align}
Plugging (\ref{suppfun-eq7})  and (\ref{divfun-eq2}) into (\ref{meancur-eq0}) we have (\ref{meancurv}).
\end{proof}

Next, we are attempted to derive the uniform upper bound of the mean curvature $H$.
\begin{theorem}
	Assume that $(M^{n+1},\bar{g})$ is a Riemannian manifold endowed with a complete conformal vector field satisfying all the conditions in Theorem \ref{convergence-0}.
 Let $\Sigma_0$ be a star-shaped, closed hypersurface embedded in $M'$, then along the flow (\ref{flow}), there exists a uniform constant $C$ independent of the time $t$ such that 
	\begin{align*}
		\max_{t\geq 0}H\leq C.
	\end{align*}
\end{theorem}
\begin{proof}
	Consider the test function 
	\begin{align*}
		\phi=H+\frac{\vert\xi\vert^2}{\varphi^2},
	\end{align*}
  from the evolution equations (\ref{meancurv}) and (\ref{C0est-eq1}) of $H$ and $\frac{\vert\xi\vert^2}{\varphi^2}$	we have
	\begin{align}\label{meancurthm-eq6}
		\partial_t\phi=&\partial_tH+\partial_t\frac{\vert\xi\vert^2}{\varphi^2}\nonumber\\
		=&u\Delta_g\phi+2\langle\nabla u,\nabla \phi\rangle+H\langle\xi,\nabla \phi\rangle-2\langle\nabla u,\nabla \frac{\vert\xi\vert^2}{\varphi^2}\rangle-H\langle\xi,\nabla \frac{\vert\xi\vert^2}{\varphi^2}\rangle\nonumber\\
  &+\varphi H^2-n\varphi\vert A\vert^2+n\varphi\left(\overline{Ric}(\mathcal{N},\mathcal{N})-\overline{Ric}(\nu,\nu)\right)\nonumber\\
		&+n\left(\langle\overline{\nabla}_{\nu}\overline{\nabla}\varphi,\nu\rangle-\langle\overline{\nabla}_{\mathcal{N}}\overline{\nabla}\varphi,\mathcal{N}\rangle\right)+4\frac{\varphi^2-\xi(\varphi)}{\varphi^4}e_i(\varphi)\langle\xi,e_i\rangle u\\
		&-2\frac{u}{\varphi^2\vert\xi\vert^2}\xi\left(\frac{\varphi^2-\xi(\varphi)}{\varphi}\right)(\vert\xi\vert^2-u^2).\nonumber
	\end{align}
	By (\ref{secconformal-eq6}) we have 
	\begin{align}\label{meancurthm-eq1}
		\nabla \frac{\vert\xi\vert^2}{\varphi^2}=2\frac{\varphi^2-\xi(\varphi)}{\varphi^3}\langle\xi,e_i\rangle e_i,
	\end{align}
 then
\begin{align}\label{meancurthm-eq4}
	-H\langle\xi,\nabla \frac{\vert\xi\vert^2}{\varphi^2}\rangle=&-2H\frac{\varphi^2-\xi(\varphi)}{\varphi^3}(\vert\xi\vert^2-u^2).
\end{align}
	Combining (\ref{meancurthm-eq1}) with (\ref{suppfun-eq1}) we obtain 
	\begin{align}\label{meancurthm-eq5}
		-2\langle\nabla u,\nabla \frac{\vert\xi\vert^2}{\varphi^2}\rangle=&-4\frac{\varphi^2-\xi(\varphi)}{\varphi^3}\langle\xi,e_i\rangle e_i(u)\nonumber\\
		=&-4\frac{\varphi^2-\xi(\varphi)}{\varphi^3}\langle\xi,e_i\rangle\left(u\frac{\langle\overline{\nabla}\varphi,e_i\rangle}{\varphi}-\frac{\langle\xi,e_i\rangle}{\varphi}\langle\overline{\nabla}\varphi,\nu\rangle+h_{ik}\langle\xi,e_k\rangle\right)\nonumber\\
		=&-4\frac{\varphi^2-\xi(\varphi)}{\varphi^4}e_i(\varphi)\langle\xi,e_i\rangle u+4\frac{\varphi^2-\xi(\varphi)}{\varphi^4}\nu(\varphi)(\vert\xi\vert^2-u^2)\\
		&-4\frac{\varphi^2-\xi(\varphi)}{\varphi^3}h_{ik}\langle\xi,e_i\rangle\langle\xi,e_k\rangle.\nonumber
	\end{align}
Thus, plugging (\ref{meancurthm-eq4}) and (\ref{meancurthm-eq5}) into (\ref{meancurthm-eq6}), we have
	\begin{align}\label{meancurthm-eq2}
		\partial_t\phi=&u\Delta_g\phi+2\langle\nabla u,\nabla \phi\rangle+H\langle\xi,\nabla \phi\rangle+\varphi H^2-n\varphi\vert A\vert^2\nonumber\\
  &-4\frac{\varphi^2-\xi(\varphi)}{\varphi^3}h_{ik}\langle\xi,e_i\rangle\langle\xi,e_k\rangle+n\varphi\left(\overline{Ric}(\mathcal{N},\mathcal{N})-\overline{Ric}(\nu,\nu)\right)\nonumber\\
		&+n\left(\langle\overline{\nabla}_{\nu}\overline{\nabla}\varphi,\nu\rangle-\langle\overline{\nabla}_{\mathcal{N}}\overline{\nabla}\varphi,\mathcal{N}\rangle\right)-2H\frac{\varphi^2-\xi(\varphi)}{\varphi^3}(\vert\xi\vert^2-u^2)\nonumber\\
		&-2\frac{u}{\varphi^2\vert\xi\vert^2}\xi\left(\frac{\varphi^2-\xi(\varphi)}{\varphi}\right)(\vert\xi\vert^2-u^2)+4\frac{\varphi^2-\xi(\varphi)}{\varphi^4}\nu(\varphi)(\vert\xi\vert^2-u^2).
	\end{align}
By the condition (iv) in Theorem \ref{convergence-0} we have
\begin{align*}
	\overline{Ric}(\mathcal{N},\mathcal{N})-\overline{Ric}(\nu,\nu)\leq 0.
\end{align*}
Note that by (\ref{curvature-eq1}), we obtain 
\begin{align}\label{meancurthm-eq3}
	\langle\overline{\nabla}_{\nu}\overline{\nabla}\varphi,\nu\rangle=-\varphi\bar{R}(\nu,\mathcal{N},\nu,\mathcal{N})+\frac{u^2}{\vert\xi\vert^2}\langle\overline{\nabla}_{\mathcal{N}}\overline{\nabla}\varphi,\mathcal{N}\rangle.
\end{align}
Plugging these into (\ref{meancurthm-eq2}), we get
\begin{align*}
	\partial_t\phi\leq& u\Delta_g\phi+2\langle\nabla u,\nabla \phi\rangle+H\langle\xi,\nabla \phi\rangle-\varphi(n\vert A\vert^2-H^2)\\
 &-4\frac{\varphi^2-\xi(\varphi)}{\varphi^3}h_{ik}\langle\xi,e_i\rangle\langle\xi,e_k\rangle-n\langle\overline{\nabla}_{\mathcal{N}}\overline{\nabla}\varphi,\mathcal{N}\rangle(1-\frac{u^2}{\vert\xi\vert^2})\\
	&-n\varphi\bar{R}(\nu,\mathcal{N},\nu,\mathcal{N})-2H\frac{\varphi^2-\xi(\varphi)}{\varphi^3}(\vert\xi\vert^2-u^2)\\
	&-2\frac{u}{\varphi^2\vert\xi\vert^2}\xi\left(\frac{\varphi^2-\xi(\varphi)}{\varphi}\right)(\vert\xi\vert^2-u^2)+4\frac{\varphi^2-\xi(\varphi)}{\varphi^4}\nu(\varphi)(\vert\xi\vert^2-u^2).
\end{align*}
	At the maximum point $p$ of $\phi$, assume that $\lambda_1$ is the smallest principle curvature, then
	\begin{align*}
		h_{ik}\langle\xi,e_i\rangle\langle\xi,e_k\rangle\geq \lambda_1(\vert\xi\vert^2-u^2).
	\end{align*}
Moreover, since $\Sigma_t$ is contained in a compact subset by Remark \ref{C0remark} , we know that $\overline{\nabla}\varphi$, $\overline{\nabla}^2\varphi$ and the sectional curvature of $M$ are bounded. Meanwhile, there are uniform positive lower bounds for $\varphi$ and $\varphi^2-\xi(\varphi)$. Thus, applying Proposition \ref{C0est-2}, at $p$ we have
	\begin{align*}
		\partial_t\phi-u\Delta_g\phi\leq &-\varphi(n\vert A\vert^2-H^2)-2\frac{\varphi^2-\xi(\varphi)}{\varphi^3}\left(H+2\lambda_1-K_1\right)(\vert\xi\vert^2-u^2),
	\end{align*}
	where $K_1$ is a uniform constant independent of the time $t$. Without loss of generality, we assume $\frac{H}{2}-K_1\geq 0$, otherwise the proof is done. This reduces the above inequality to
	\begin{align*}
		\partial_t\phi-u\Delta_g\phi\leq &-\varphi(n\vert A\vert^2-H^2)-2\frac{\varphi^2-\xi(\varphi)}{\varphi^3}\left(\frac{H}{2}+2\lambda_1\right)(\vert\xi\vert^2-u^2).
	\end{align*}
Now, we consider two different cases.
	\item[Case 1.] Suppose that $\frac{H}{2}+2\lambda_1\geq 0$ at $p$. Then
	\begin{align*}
		\partial_t\phi-u\Delta_g\phi\leq 0,
	\end{align*}
	and by maximum principle we see that $\phi$ is bounded above, and so is $H$.
	\item[Case 2.] Suppose that $\frac{H}{2}+2\lambda_1<0$ at $p$. Let $\{\lambda_2,\cdots,\lambda_n\}$ be  other principle curvatures at the point $p$, $H_*=\lambda_2+\cdots+\lambda_n$ and $\vert A_*\vert^2=\lambda_2^2+\cdots+\lambda_n^2$. Then 
	\begin{align*}
		-\lambda_1\geq \frac{H}{4}
	\end{align*}
	and
	\begin{align*}
		n\vert A\vert^2-H^2=&n\vert A_*\vert^2-H_*^2-2\lambda_1H+(n+1)\lambda_1^2\\
		\geq& (n+1)\lambda_1^2-2\lambda_1H.
	\end{align*}
	Recall that there are positive lower bounds for $\varphi$ and  $\varphi^2-\xi(\varphi)$, and $\frac{\vert\xi\vert^2-u^2}{\varphi^2}$ has a uniform upper bound by Proposition \ref{C0est-2}. These yield that there are uniform positive constant $K_2$ and $K_3$ such that 
	\begin{align*}
		\partial_t\phi-u\Delta_g\phi\leq
		&-(n+1)K_2\lambda^2_1-K_3H+2\lambda_1(K_2H-2K_3).
	\end{align*}
	We assume that $H\geq\frac{K_3}{2K_2}>0$, otherwise the proof is done. Hence $-\lambda_1\geq 0$ and then
	\begin{align*}
		\partial_t\phi-u\Delta_g\phi\leq 0.
	\end{align*}
	By the maximum principle, we conclude that $\phi$ has a uniform upper bound, and so is $H$.
\end{proof}

By the uniform upper bound of $H$ and the uniform $C^0$ estimate (Proposition \ref{C0est-2}), the terms on the right hand of (\ref{ev-support fun}) tend to zero as $u\to 0$, except the term $n(\varphi^2-\xi(\varphi))$. Thus by the maximum principle we have shown the uniform positivity of the support function $u$.
\begin{proposition}\label{C1-est}
	Assume that $(M^{n+1},\bar{g})$ is a Riemannian manifold endowed with a complete conformal vector field satisfying all the conditions in Theorem \ref{convergence-0}. Let $\Sigma_0$ be a star-shaped, closed hypersurface embedded in $M'$.
	Then along the flow (\ref{flow}), there exists a constant $\epsilon>0$ independent of the time $t$ such that 
	\begin{align*}
		\min_{t\geq 0}u\geq \epsilon.
	\end{align*}
\end{proposition}

Now, we are prepared to prove the long time existence and convergence of the volume preserving flow (\ref{flow}).
\begin{theorem}\label{convergence}
	Assume that  $(M^{n+1},\bar{g})$  is a Riemannian manifold endowed with a complete conformal vector field satisfying all the conditions in Theorem \ref{convergence-0}. Let $\Sigma_0$ be a star-shaped, closed hypersurface embedded in $M'$. Then the evolution equation (\ref{flow}) with $\Sigma_0$ as a initial data has a smooth solution for $t\in[0,+\infty)$. Moreover, the solution hypersurfaces converge to a  totally umbilical hypersurface $\Sigma_{\infty}$ whose unit normal vector field $\nu_{\infty}$ attains least Ricci curvature on $M'$, that means
	\begin{align*}
		\overline{Ric}(\nu_{\infty},\nu_{\infty})=\overline{Ric}(\mathcal{N},\mathcal{N}).
	\end{align*}
\end{theorem}
\begin{proof}
	Since the evolution equation (\ref{C0est-eq1}) of the function $\frac{\vert\xi\vert^2}{\varphi^2}$ is a quasilinear parabolic equation, it follows from Proposition \ref{C0est-2} and Proposition \ref{C1-est} that the equation is uniformly parabolic. Then the regularity of function $\frac{\vert\xi\vert^2}{\varphi^2}$ follows from the standard parabolic theory. Observing that
	\begin{align*}
		\nabla_i\frac{\vert\xi\vert^2}{\varphi^2} =2\frac{\varphi^2-\xi(\varphi)}{\varphi^3}\langle\xi,e_i\rangle,
	\end{align*} 
and 
\begin{align*}
	\nabla_j\nabla_i\frac{\vert\xi\vert^2}{\varphi^2} =&\frac{2}{\varphi^2\vert\xi\vert^2}\xi\left(\frac{\varphi^2-\xi(\varphi)}{\varphi}\right)\langle\xi,e_i\rangle\langle\xi,e_j\rangle -2\frac{\varphi^2-\xi(\varphi)}{\varphi^4}e_i(\varphi)\langle\xi,e_j\rangle \\
	&-2\frac{\varphi^2-\xi(\varphi)}{\varphi^4}e_j(\varphi)\langle\xi,e_i\rangle+2\frac{\varphi^2-\xi(\varphi)}{\varphi^2}g_{ij}-2\frac{\varphi^2-\xi(\varphi)}{\varphi^3}uh_{ij}.
\end{align*}
Since $\varphi>0$, $\varphi^2-\xi(\varphi)>0$ and $u\geq\epsilon$, the regularity of $\frac{\vert\xi\vert^2}{\varphi^2}$ implies the uniform estimates of the second fundamental form and its higher order derivatives. Then we can conclude the long-time existence and convergence of the flow (\ref{flow}). We denote the limit hypersurface by $\Sigma_{\infty}$, then $\Sigma_{\infty}\in M'$. Recall that along the flow,
\begin{align*}
		A'(t)
	&=\int_{\Sigma}\left(\frac{2n}{n-1}\sigma_2-H^2+\frac{n}{n-1}(\overline{Ric}(\mathcal{N},\mathcal{N})-\overline{Ric}(\nu,\nu))\right)ud\mu.
\end{align*}
Since
\begin{align*}
	\int_0^{\infty}A'(t)d\mu=A(\Sigma_{\infty})-A(\Sigma_0)<+\infty,
\end{align*}
this yields $\lim_{t\to+\infty} A'(t)=0$ by $A'(t)\leq 0$. Thus,
\begin{align*}
	\int_{\Sigma_{\infty}}\left(\frac{2n}{n-1}\sigma_2-H^2+\frac{n}{n-1}(\overline{Ric}(\mathcal{N},\mathcal{N})-\overline{Ric}(\nu_{\infty},\nu_{\infty}))\right)ud\mu=0.
\end{align*}
Therefore, on $\Sigma_{\infty}$ we have:
\begin{align*}
	\overline{Ric}(\mathcal{N},\mathcal{N})=\overline{Ric}(\nu_{\infty},\nu_{\infty})
\end{align*}
and 
\begin{align*}
	\frac{2n}{n-1}\sigma_2=H^2,
\end{align*} 
the equality holds in the Newton-Maclaurin inequality implies that $\Sigma_{\infty}$ is totally umbilical.
\end{proof}
Then,  the following corollary holds evidently.
\begin{corollary}\label{strictricci}
	Under the same conditions as that in Theorem \ref{convergence}, if we assume extra that our hypothesis on the Ricci curvature of $M$ is strict, that is, the $\mathcal{N}$ direction is the only direction of least Ricci curvature. Then, $\Sigma_{\infty}$ is a leaf of the foliation $\mathcal{F}(\xi)$, i.e it's a level set hypersurface of the function $\frac{\vert\xi\vert}{\varphi}$.
\end{corollary}
Recall Theorem \ref{monotonethm}, the area of evolving hypersurface along flow (\ref{flow}) is non-increasing and the enclosed volume is preserved. This proves the following isoperimetric inequality:
\begin{theorem}
		Assume that $(M^{n+1},\bar{g})$ is a Riemannian manifold endowed with a complete conformal vector field satisfying all the conditions in Theorem \ref{convergence-0}. Moreover, we assume extra that our hypothesis on the Ricci curvature of $M$ is strict. Let $\Sigma\subset M'$ be a star-shaped, closed hypersurface which satisfies $\min_{\Sigma}\frac{\vert\xi\vert}{\varphi}>r_1$, and $\Omega$ be the bounded domain enclosed by $\Sigma$ and $S(r_1)$, then 
		\begin{align}\label{isoineq-1}
			Area(\Sigma)\geq Area(S(r^*)),
		\end{align}
	where $r^*$ is the unique real number determined by $\text{Vol}(B(r^*))=\text{Vol}(\Omega)$, and $B(r^*)$ is the domain enclosed by $S(r_1)$ and $S(r^*)$. Moreover, "$=$" attains in (\ref{isoineq-1}) if and only if $\Sigma=S(r^*)$.
\end{theorem}
\begin{proof}
	By Theorem \ref{monotonethm} and Corollary \ref{strictricci}, the inequality (\ref{isoineq-1}) holds.
	
	 On the other hand, if $\Sigma$ is a star-shaped, closed hypersurface which satisfies $$Area(\Sigma)=Area(S(r^*)),$$
	  but $\Sigma$ is not a leaf of the foliation $\mathcal{F}(\xi)$. Then $\overline{Ric}(\nu,\nu)<\overline{Ric}(\mathcal{N},\mathcal{N})$ on $\Sigma$ by the assumption. Along the flow (\ref{flow}), the formula (\ref{mono-eq1}) implies $A'(0)<0$, this is a contradiction to $Area(\Sigma)=Area(S(r^*))$.
\end{proof}

\section{A solution to the isoperimetric problem}\label{sol to isoper}

In this section, we will prove that the hypersurface $\Sigma_{\infty}$ in Theorem \ref{convergence} is a leaf of the foliation $\mathcal{F}(\xi)$ even without the extra assumption on the Ricci curvature of $M$ in Corollary \ref{strictricci}. As a conclusion, we can show that the level sets of the function $\frac{\vert\xi\vert}{\varphi}$ are solutions of the isoperimetric problem at this time. To achieve this end, we need to assume that the sectional curvature $K(X,\xi)=\frac{\bar{R}(X,\xi,X,\xi)}{\vert X\vert^2\vert\xi\vert^2-\langle X,\xi\rangle^2}$ on $M'$ satisfies 
	\begin{align}\label{limit-eq0}
	K(X,\xi)\geq -\frac{\varphi^2}{\vert\xi\vert^2},\ \forall X\in TM,\ X\neq\xi.
\end{align}
For any compact leaf $P$ of the foliation $\mathcal{F}(\xi)$ on $M'$, using the local normal coordinates $\{\tilde{e}_1,\cdots,\tilde{e}_n\}$ on $P$, by the Gauss equation we have
\begin{align*}
	\tilde{R}_{ijk\ell}=\bar{R}_{ijk\ell}+\frac{\varphi^2}{\vert\xi\vert^2}(\bar{g}_{ik}\bar{g}_{j\ell}-\bar{g}_{i\ell}\bar{g}_{jk}),
\end{align*}
where $\tilde{R}_{ijk\ell}$ is the components of the curvature tensor of $P$. Then, we get the Ricci curvature of $P$ as follow:
\begin{align*}
	\tilde{Ric}(\tilde{e}_i,\tilde{e}_i)=\overline{Ric}(\tilde{e}_i,\tilde{e}_i)-\bar{R}(\tilde{e}_i,\mathcal{N},\tilde{e}_i,\mathcal{N})+(n-1)\frac{\varphi^2}{\vert\xi\vert^2}.
\end{align*}
Therefore, using the condition (iv) in Theorem \ref{convergence-0} and (\ref{limit-eq0}) we have
\begin{align*}
	\tilde{Ric}(\tilde{e}_i,\tilde{e}_i)\geq&\overline{Ric}(\mathcal{N},\mathcal{N})-\bar{R}(\tilde{e}_i,\mathcal{N},\tilde{e}_i,\mathcal{N})+(n-1)\frac{\varphi^2}{\vert\xi\vert^2}\\
	=&\sum_{j\neq i}\bar{R}(\tilde{e}_j,\mathcal{N},\tilde{e}_j,\mathcal{N})+(n-1)\frac{\varphi^2}{\vert\xi\vert^2}
	\geq0.
\end{align*}
So,  the condition (\ref{limit-eq0})  implies that the Ricci curvatures of the leaf of $\mathcal{F}(\xi)$ are nonnegative.

By Theorem \ref{convergence}, we have proved that the limit hypersurface $\Sigma_{\infty}$ is a smooth totally umbilical hypersurface in $M'$.
In order to prove that $\Sigma_{\infty}$ is a  leaf of the foliation $\mathcal{F}(\xi)$, we only need to prove  $\vert\xi\vert^2=u^2$ on $\Sigma_{\infty}$. We shall consider the auxiliary function
\[\eta=\frac{\vert\xi\vert^2}{u^2}-1,\]
and attempt to prove that $\eta$ decays to zero along the flow (\ref{flow}).
\begin{proposition}
		Assume that $(M^{n+1},\bar{g})$ is a  Riemannian manifold endowed with a complete conformal vector field satisfying the conditions (i) and (iii) in Theorem \ref{convergence-0}. Let $\Sigma_0$ be a closed hypersurface embedded in $M'$, then under the flow (\ref{flow}), the norm of the conformal vector field $\vert\xi\vert^2$ evolves by 
	\begin{align}\label{norm}
		\partial_t\vert\xi\vert^2=&u\Delta_g\vert\xi\vert^2-2\frac{u}{\varphi}\langle\nabla\varphi,\nabla\vert\xi\vert^2\rangle+2n\vert\xi\vert^2\nu(\varphi)+2\frac{\vert\nabla\varphi\vert^2}{\varphi^2}u\vert\xi\vert^2\\&-2\frac{\vert\overline{\nabla}\varphi\vert^2}{\varphi^2}u(\vert\xi\vert^2-u^2)+2u\overline{Ric}(\xi,\xi)-2u\bar{R}(\nu,\xi,\nu,\xi).\nonumber
	\end{align}
\end{proposition}
\begin{proof}
	First, along the flow (\ref{flow}), by (\ref{1stder}) we have
	\begin{align}\label{norm-eq1}
		\partial_t\vert\xi\vert^2=&2(n\varphi-uH)\langle\overline{\nabla}_{\nu}\xi,\xi\rangle\nonumber\\
		=&2\frac{\varphi^2-\xi(\varphi)}{\varphi}(n\varphi-uH)u+2\frac{\vert\xi\vert^2\nu(\varphi)}{\varphi}(n\varphi-uH).
	\end{align}
Next, using (\ref{1stder}) again, we have
\begin{align*}
	\nabla_i\vert\xi\vert^2=&2\langle\overline{\nabla}_i\xi,\xi\rangle=2\frac{\varphi^2-\xi(\varphi)}{\varphi}\langle\xi,e_i\rangle+2\frac{\langle\overline{\nabla}\varphi,e_i\rangle}{\varphi}\vert\xi\vert^2.
\end{align*}
Then
\begin{align}\label{norm-eq2}
	\Delta_g\vert\xi\vert^2
	=&2\frac{\varphi^2-\xi(\varphi)}{\varphi}(n\varphi-uH)+2\langle\xi,e_i\rangle e_i(\frac{\varphi^2-\xi(\varphi)}{\varphi})+2\frac{\langle\nabla\varphi,\nabla\vert\xi\vert^2\rangle}{\varphi}\nonumber\\
	&-2\frac{\nu(\varphi)}{\varphi}\vert\xi\vert^2H-2\frac{\vert\nabla\varphi\vert^2}{\varphi^2}\vert\xi\vert^2+2\frac{\vert\xi\vert^2}{\varphi}\langle\overline{\nabla}_{e_i}\overline{\nabla}\varphi,e_i\rangle.
\end{align}
Now, by (\ref{secconformal-eq7}) we get that
\begin{align}\label{norm-eq3}
	2\langle\xi,e_i\rangle e_i(\frac{\varphi^2-\xi(\varphi)}{\varphi})=&2\xi(\frac{\varphi^2-\xi(\varphi)}{\varphi})(1-\frac{u^2}{\vert\xi\vert^2})\nonumber\\
	=&2\left(\xi(\varphi)+\frac{(\xi(\varphi))^2}{\varphi^2}-\frac{\xi(\xi(\varphi))}{\varphi}\right)(1-\frac{u^2}{\vert\xi\vert^2}).
\end{align}
Then, we used (\ref{1stder}) to compute
\begin{align}\label{norm-eq6}
	\langle\overline{\nabla}_{\mathcal{N}}\overline{\nabla}\varphi,\mathcal{N}\rangle=&\frac{1}{\vert\xi\vert^2}\left(\xi(\xi(\varphi))-\langle\overline{\nabla}\varphi,\overline{\nabla}_{\xi}\xi\rangle\right)\nonumber\\
	=&\frac{1}{\vert\xi\vert^2}\left(\xi(\xi(\varphi))-\varphi\xi(\varphi)-\frac{(\xi(\varphi))^2}{\varphi}+\frac{\vert\xi\vert^2}{\varphi}\vert\overline{\nabla}\varphi\vert^2\right).
\end{align} 
Combining (\ref{norm-eq3}) with (\ref{norm-eq6}) we have
\begin{align}\label{norm-eq7}
	2\langle\xi,e_i\rangle e_i(\frac{\varphi^2-\xi(\varphi)}{\varphi})=2\left(\frac{\vert\xi\vert^2}{\varphi^2}\vert\overline{\nabla}\varphi\vert^2-\frac{\vert\xi\vert^2}{\varphi}\langle\overline{\nabla}_{\mathcal{N}}\overline{\nabla}\varphi,\mathcal{N}\rangle\right)(1-\frac{u^2}{\vert\xi\vert^2}).
\end{align}
On the other hand, recall (\ref{curvature-eq2}) we get
\begin{align}\label{norm-eq4}
	2\frac{\vert\xi\vert^2}{\varphi}\langle\overline{\nabla}_{e_i}\overline{\nabla}\varphi,e_i\rangle=&2\frac{\vert\xi\vert^2}{\varphi}\left(\overline{\Delta}\varphi-\langle\overline{\nabla}_{\nu}\overline{\nabla}\varphi,\nu\rangle\right)\nonumber\\
	=&-2\overline{Ric}(\xi,\xi)+2\frac{\vert\xi\vert^2}{\varphi}\left(\langle\overline{\nabla}_{\mathcal{N}}\overline{\nabla}\varphi,\mathcal{N}\rangle-\langle\overline{\nabla}_{\nu}\overline{\nabla}\varphi,\nu\rangle\right).
\end{align}
Plugging (\ref{norm-eq7}) and (\ref{norm-eq4}) into (\ref{norm-eq2}), we derive that
\begin{align}\label{norm-eq5}
	\Delta_g\vert\xi\vert^2=&2\frac{\langle\nabla\varphi,\nabla\vert\xi\vert^2\rangle}{\varphi}+2\frac{\varphi^2-\xi(\varphi)}{\varphi}(n\varphi-uH)+2\frac{\vert\overline{\nabla}\varphi\vert^2}{\varphi^2}(\vert\xi\vert^2-u^2)\nonumber\\
	&-2\frac{\vert\nabla\varphi\vert^2}{\varphi^2}\vert\xi\vert^2-2\frac{\nu(\varphi)}{\varphi}\vert\xi\vert^2H-2\overline{Ric}(\xi,\xi)+2\bar{R}(\nu,\xi,\nu,\xi),
\end{align}
where we used (\ref{curvature-eq1}). Now, combining (\ref{norm-eq1}) with (\ref{norm-eq5}) we obtain (\ref{norm}).

\end{proof}

\begin{proposition}
	Assume that $(M^{n+1},\bar{g})$ is a  Riemannian manifold endowed with a complete conformal vector field satisfying the conditions (i) and (iii) in Theorem \ref{convergence-0}. Let $\Sigma_0$ be a star-shaped, closed hypersurface embedded in $M'$, then under the flow (\ref{flow}), the function $\eta=\frac{\vert\xi\vert^2}{u^2}-1$ evolves by
	\begin{align}\label{fun2}
		\partial_t\eta=&u\Delta_g\eta+4\langle\nabla u,\nabla\eta\rangle+\left(H-2\frac{\varphi^2-\xi(\varphi)}{\varphi}\frac{u}{\vert\xi\vert^2}\right)\langle\nabla\eta,\xi\rangle\nonumber\\
		&-2\frac{u}{\varphi}\langle\nabla\varphi,\nabla\eta\rangle+\frac{u^3\vert\nabla\eta\vert^2}{2\vert\xi\vert^2}+2\left(\frac{\varphi^2-\xi(\varphi)}{\varphi}+\frac{\vert\xi\vert^2}{u}\frac{\nu(\varphi)}{\varphi}+\varphi\frac{\vert\xi\vert^2}{u^2}\right)H\nonumber\\
  &-2\frac{\vert\xi\vert^2}{u}\vert A\vert^2+2\frac{(\varphi^2-\xi(\varphi))^2}{\varphi^2}\frac{u}{\vert\xi\vert^2}\eta-2\frac{\vert\overline{\nabla}\varphi\vert^2}{\varphi^2}u\eta-2n(\varphi^2-\xi(\varphi))\frac{\vert\xi\vert^2}{u^3}\\
		&-2n\frac{\vert\xi\vert^2}{u^2}\nu(\varphi)-2u^{-1}\bar{R}(\nu,\xi,\nu,\xi)+2\frac{\vert\xi\vert^2}{u}\left(\overline{Ric}(\mathcal{N},\mathcal{N})-\overline{Ric}(\nu,\nu)\right).\nonumber
	\end{align}
\end{proposition}
\begin{proof}
By the evolution equations (\ref{ev-support fun}) and (\ref{norm}) of the support function $u$ and the function $\vert\xi\vert^2$, we have
	\begin{align}\label{fun2-eq1}
		\partial_t\eta=&u^{-2}\partial_t\vert\xi\vert^2-2\frac{\vert\xi\vert^2}{u^3}\partial_tu\nonumber\\
		=&u^{-1}\Delta_g\vert\xi\vert^2-2\frac{\vert\xi\vert^2}{u^2}\Delta_gu-2\frac{1}{u\varphi}\langle\nabla\varphi,\nabla\vert\xi\vert^2\rangle-2\frac{\vert\xi\vert^2}{u^3}H\langle\xi,\nabla u\rangle\nonumber\\
		&+2\frac{\vert\nabla\varphi\vert^2}{\varphi^2}\frac{\vert\xi\vert^2}{u}-2\frac{\vert\overline{\nabla}\varphi\vert^2}{\varphi^2}\frac{\vert\xi\vert^2-u^2}{u}-2\frac{\vert\xi\vert^2}{u}\vert A\vert^2+4\varphi\frac{\vert\xi\vert^2}{u^2}H\nonumber\\
  &-2n(\varphi^2-\xi(\varphi))\frac{\vert\xi\vert^2}{u^3}-2n\frac{\vert\xi\vert^2}{u^2}\nu(\varphi)-2u^{-1}\bar{R}(\nu,\xi,\nu,\xi)\\
		&+2\frac{\vert\xi\vert^2}{u}\left(\overline{Ric}(\mathcal{N},\mathcal{N})-\overline{Ric}(\nu,\nu)\right).\nonumber
	\end{align}
Since
\begin{align}\label{fun2-eq2}
	\nabla\eta=u^{-2}\nabla\vert\xi\vert^2-2\frac{\vert\xi\vert^2}{u^3}\nabla u,
\end{align}
we get 
\begin{align*}
	u\Delta_g\eta=u^{-1}\Delta_g\vert\xi\vert^2-2\frac{\vert\xi\vert^2}{u^2}\Delta_gu-4\langle\nabla u,\nabla\eta\rangle-2\frac{\vert\xi\vert^2}{u^3}\vert\nabla u\vert^2,
\end{align*}
and
\begin{align*}
	-2\frac{1}{u\varphi}\langle\nabla\varphi,\nabla\vert\xi\vert^2\rangle=-2\frac{u}{\varphi}\langle\nabla\varphi,\nabla\eta\rangle-4\frac{\vert\xi\vert^2}{u^2\varphi}\langle\nabla u,\nabla \varphi\rangle.
\end{align*}
Plugging these into (\ref{fun2-eq1}) we have
\begin{align}\label{fun2-eq3}
		\partial_t\eta=&	u\Delta_g\eta+4\langle\nabla u,\nabla\eta\rangle-2\frac{u}{\varphi}\langle\nabla\varphi,\nabla\eta\rangle+H\langle\nabla\eta,\xi\rangle\nonumber\\
  &+2\frac{\vert\xi\vert^2}{u^3}\vert\nabla u-\frac{u}{\varphi}\nabla\varphi\vert^2-\frac{H}{u^2}\langle\xi,\nabla\vert\xi\vert^2\rangle-2\frac{\vert\xi\vert^2}{u}\vert A\vert^2+4\varphi\frac{\vert\xi\vert^2}{u^2}H\\
		&-2n\frac{\vert\xi\vert^2}{u^2}\nu(\varphi)-2\frac{\vert\overline{\nabla}\varphi\vert^2}{\varphi^2}u\eta-2n(\varphi^2-\xi(\varphi))\frac{\vert\xi\vert^2}{u^3}\nonumber\\
  &+2\frac{\vert\xi\vert^2}{u}\left(\overline{Ric}(\mathcal{N},\mathcal{N})-\overline{Ric}(\nu,\nu)\right)-2u^{-1}\bar{R}(\nu,\xi,\nu,\xi).\nonumber
\end{align}
Moreover, using (\ref{1stder}), we continue to compute $\nabla\eta$ after (\ref{fun2-eq2}),
\begin{align}\label{fun2-eq4}
	\nabla\eta
	=&2\frac{\vert\xi\vert^2}{u^3}\left(\frac{u}{\varphi}\nabla\varphi-\nabla u\right)+2\frac{\varphi^2-\xi(\varphi)}{\varphi}\frac{\xi^{\top}}{u^2},
\end{align} 
this yields
\begin{align}\label{fun2-eq5}
	2\frac{\vert\xi\vert^2}{u^3}\vert\nabla u-\frac{u}{\varphi}\nabla\varphi\vert^2=&\frac{u^3}{2\vert\xi\vert^2}\vert\nabla\eta\vert^2-2\frac{\varphi^2-\xi(\varphi)}{\varphi}\frac{u}{\vert\xi\vert^2}\langle\nabla\eta,\xi\rangle\nonumber\\
 &+2\frac{(\varphi^2-\xi(\varphi))^2}{\varphi^2}\frac{u}{\vert\xi\vert^2}\eta.
\end{align}
Meanwhile, by (\ref{1stder}) we have
\begin{align}\label{fun2-eq6}
	\langle\xi,\nabla\vert\xi\vert^2\rangle=2\frac{\varphi^2-\xi(\varphi)}{\varphi}(\vert\xi\vert^2-u^2)+2\frac{\vert\xi\vert^2}{\varphi}(\xi(\varphi)-u\nu(\varphi)).
\end{align}
Substituting (\ref{fun2-eq5}) and (\ref{fun2-eq6}) into (\ref{fun2-eq3}) we obtain (\ref{fun2}).
\end{proof}
\begin{theorem}\label{limit surface}
	Assume that $(M^{n+1},\bar{g})$ is a Riemannian manifold which satisfies all the conditions in Theorem \ref{convergence-0}. Moreover, assume that  the sectional curvature of the direction  $\xi$ satisfies (\ref{limit-eq0}) on $M'$.
Let $\Sigma_0$ be a star-shaped, closed hypersurface embedded in $M'$. Then, the solution hypersurfaces to the flow (\ref{flow}) with the initial data $\Sigma_0$ converge to a leaf of the foliation $\mathcal{F}(\xi)$.
\end{theorem}
\begin{proof}
	We focus on the evolution equation (\ref{fun2}) of the auxiliary function $\eta$. At the maximum point $p$ of $\eta$ on $\Sigma_t$,  without loss of generality, we assume that $\eta(p,t)> 0$, otherwise the proof is done.
 Recall (\ref{fun2-eq4}), at the point $p$ we have 
	\begin{align*}
		\nabla u-\frac{u\nabla\varphi}{\varphi}=\frac{\varphi^2-\xi(\varphi)}{\varphi}\frac{u}{\vert\xi\vert^2}\xi^{\top},
	\end{align*}
this implies that 
	\begin{align*}
		h_{ik}\langle\xi,e_k\rangle=&\frac{\varphi^2-\xi(\varphi)}{\varphi} \frac{u}{\vert\xi\vert^2}\langle\xi,e_i\rangle+\frac{\nu(\varphi)}{\varphi}\langle\xi,e_i\rangle,
	\end{align*}
where we used (\ref{suppfun-eq1}).
	We can choose the orthonormal frame $\{e_1=\frac{\xi^{\top}}{\vert\xi^{\top}\vert},e_2,\cdots,e_n\}$ around the point $p$, then
	\begin{align*}
		h_{11}=\frac{\varphi^2-\xi(\varphi)}{\varphi}\frac{u}{\vert\xi\vert^2}+\frac{\nu(\varphi)}{\varphi}.
	\end{align*}
	Taking $i\neq 1$ we have $h_{i1}=0$. Hence ,  $h_{11}$ is one of the principle curvatures of $\Sigma_t$ at the point $p$, then we have $$H=\frac{\varphi^2-\xi(\varphi)}{\varphi}\frac{u}{\vert\xi\vert^2}+\frac{\nu(\varphi)}{\varphi}+\sum_{i=2}^{n}\lambda_i$$ and 
	$$\vert  A\vert^2=h_{11}^2+\sum_{i=2}^{n}\lambda_i^2,$$
	where $\{\lambda_2,\cdots,\lambda_n\}$ are other principle curvatures of $\Sigma_t$ at the point $p$.
	Plugging this into (\ref{fun2}), at the point $p$ we have
	\begin{align}\label{limit-eq1}
		\partial_t\eta\leq&2\frac{(\varphi^2-\xi(\varphi))^2}{\varphi^2}\frac{u}{\vert\xi\vert^2}\eta-2\frac{\vert\overline{\nabla}\varphi\vert^2}{\varphi^2}u\eta+2\left(\frac{\vert\xi\vert^2}{u}h_{11}+\varphi\frac{\vert\xi\vert^2}{u^2}\right)H\nonumber\\
  &-2\frac{\vert\xi\vert^2}{u}(h_{11}^2+\sum_{i=2}^{n}\lambda_i^2)-2n(\varphi^2-\xi(\varphi))\frac{\vert\xi\vert^2}{u^3}-2n\frac{\vert\xi\vert^2}{u^2}\nu(\varphi)\nonumber\\
		&-2u^{-1}\bar{R}(\nu,\xi,\nu,\xi)\nonumber\\
		 =&2\frac{(\varphi^2-\xi(\varphi))^2}{\varphi^2}\frac{u}{\vert\xi\vert^2}\eta-2\frac{\vert\overline{\nabla}\varphi\vert^2}{\varphi^2}u\eta-2\frac{\vert\xi\vert^2}{u}\sum_{i=2}^{n}\lambda_i^2+2\frac{\vert\xi\vert^2}{u}h_{11}\sum_{i=2}^{n}\lambda_i\nonumber\\
   &+2\varphi\frac{\vert\xi\vert^2}{u^2}H-2n(\varphi^2-\xi(\varphi))\frac{\vert\xi\vert^2}{u^3}-2n\frac{\vert\xi\vert^2}{u^2}\nu(\varphi)-2u^{-1}\bar{R}(\nu,\xi,\nu,\xi).
	\end{align}
Since we already proved that the hypersurfaces $\Sigma_t$ converge to $\Sigma_{\infty}$ which is totally umbilical, we conclude that   
$$\lambda_i=h_{11}+\delta_i(t),\quad \forall i=2,\cdots,n,$$
where $\delta_i(t)$ are functions which tend to zero as $t\to+\infty$.
Therefore, (\ref{limit-eq1}) is reduced to
\begin{align}\label{limit-eq2}
	\partial_t\eta\leq &2\frac{(\varphi^2-\xi(\varphi))^2}{\varphi^2}\frac{u}{\vert\xi\vert^2}\eta-2\frac{\vert\overline{\nabla}\varphi\vert^2}{\varphi^2}u\eta-2u^{-1}\bar{R}(\nu,\xi,\nu,\xi)\nonumber\\
 &-2n(\varphi^2-\xi(\varphi))u^{-1}\eta+\delta(t)\nonumber\\
	\leq&2(\varphi^2-2\xi(\varphi))\frac{u}{\vert\xi\vert^2}\eta-2u^{-1}\bar{R}(\nu,\xi,\nu,\xi)\\
 &-2n(\varphi^2-\xi(\varphi))u^{-1}\eta+\delta(t),\nonumber
\end{align}
where in the last inequality we used the fact that $\vert\overline{\nabla}\varphi\vert^2\geq \frac{(\xi(\varphi))^2}{\vert\xi\vert^2}$ and $\delta(t)$ is a function tends to zero as $t\to+\infty$. Next, by the assumption (\ref{limit-eq0}), we have
\begin{align*}
	\bar{R}(\nu,\xi,\nu,\xi)\geq -\varphi^2\frac{u^2}{\vert\xi\vert^2}\eta,
\end{align*}
then
\begin{align*}
	\partial_t\eta
	\leq&4(\varphi^2-\xi(\varphi))\frac{u}{\vert\xi\vert^2}\eta-2n(\varphi^2-\xi(\varphi))u^{-1}\eta+\delta(t)\\
	=&-2(n-2)(\varphi^2-\xi(\varphi))u^{-1}\eta-4(\varphi^2-\xi(\varphi))\frac{u}{\vert\xi\vert^2}\eta^2+\delta(t).
\end{align*}
Since $\varphi^2-\xi(\varphi)>0$, by Proposition \ref{C0est-2} and \ref{C1-est} there is a uniform constant $\epsilon_1$ such that
\begin{align*}
	\partial_t\eta\leq&-\epsilon_1\eta^2+\delta(t)
\end{align*}
at the maximum point $p$ of $\eta$.
Therefore, by the maximum principle the function $\eta$ decays to zero along the flow (\ref{flow}), our proof is complete.
\end{proof}
Now, combining the above theorem with Theorem \ref{monotonethm} and Theorem \ref{convergence}, we have proved the isoperimetric inequality in Theorem \ref{isoineq2}.

\section{Closed conformal vector fields}\label{closed}

In this section, we will  investigate a special class of conformal vector fields, which is closed conformal vector fields. Recall that a conformal vector field  $\xi$ is called closed if
\begin{align}\label{cconformal}
	\overline{\nabla}_{X}\xi=\varphi X,\ \forall X\in TM.
\end{align}

For the Riemannian manifold admitting a closed conformal vector field $\xi$, the distribution $\mathcal{D}$ on $M'$ is integrable. It has been proved that each leaf of the foliation $\mathcal{F}(\xi)$ has constant mean curvature. More precisely, the following property holds (see \cite[Propostion 1]{Montiel99} for details).
\begin{proposition}\label{cconformal mfld1}
	Let $(M^{n+1},\bar{g})$ be a Riemannian manifold endowed with a closed non-trivial conformal vector field $\xi$. Then we have that 
	\begin{itemize}
		\item[(1)] On the open dense set $M'$,
		\[\overline{\nabla}_{\mathcal{N}}\mathcal{N}=0,\ \overline{\nabla}_X\mathcal{N}=\frac{\varphi}{\vert\xi\vert}X \ \text{if}\  X\bot\xi.\]
		In particular, the flow of $\mathcal{N}$ is a unit speed geodesic flow.
		\item[(2)] The functions $\vert\xi\vert$, $\varphi$ and $\xi(\varphi)$ are constant on connected leaves of $\mathcal{F}(\xi)$ and each leaf has constant mean curvature $H=n\frac{\varphi}{\vert\xi\vert}$.
	\end{itemize}
\end{proposition}
Hence, by (\ref{cconformal}) and (2) in Proposition \ref{cconformal mfld1} we have the following proposition:
\begin{proposition}\label{proposition vec prop3}
	Let $(M^{n+1},\bar{g})$ be a Riemannian manifold endowed with a closed, complete conformal vector field $\xi$, then the following identities hold on the open set $M'$:
	\begin{align}
		&\overline{\nabla}\vert\xi\vert^2=2\varphi\xi, \label{iden1}\\
		&\vert\xi\vert^2\overline{\nabla}\varphi=\xi(\varphi)\xi, \label{iden4}\\
		&\vert\xi\vert^2\overline{\nabla}(\xi(\varphi))=\xi(\xi(\varphi))\xi. \label{iden5}
	\end{align}
\end{proposition}
Next, by direct computation, Proposition \ref{conformal vec prop2} can be reduced to the following simplified version.
\begin{proposition}
	Let $(M^{n+1},\bar{g})$ be a Riemannian manifold endowed with a closed, complete conformal vector field $\xi$. Then, on $M'$ we have
	\begin{align}\label{ccurvature}
		K(X,\xi)
		=&-\frac{\xi(\varphi)}{\vert\xi\vert^2},\ \forall X\in TM,X\neq\xi,
	\end{align}
	\begin{align}\label{ricci1}
		\overline{Ric}(\xi,X)=-n\frac{\xi(\varphi)}{\vert\xi\vert^2}\langle\xi,X\rangle, \ \forall X\in TM.
	\end{align}
	Moreover, 
	\begin{align}\label{ricci2}
		\overline{Ric}(\mathcal{N},\mathcal{N})=-n\frac{\xi(\varphi)}{\vert\xi\vert^2}.
	\end{align}
\end{proposition}
\begin{proof}
	Substituting (\ref{iden4}) and (\ref{iden5}) into (\ref{curvature-eq1}) and (\ref{curvature-eq2}), the proof is complete. 
\end{proof}
\begin{remark}
	From (\ref{ccurvature}), if the conformal vector field $\xi$ is closed, then the condition (\ref{limit-eq0})  for the sectional  curvature holds when $\varphi^2-\xi(\varphi)\geq 0$.
\end{remark}
Until now, if we assume that $M$ is a Riemannian manifold endowed with a closed, complete conformal vector field $\xi$ satisfies $\varphi>0$ and $\varphi^2-\xi(\varphi)>0$, and the direction determined by $\xi$ is of least Ricci curvature on $M$. Then by Theorem \ref{isoineq2}, the isoperimetric inequality  holds for any star-shaped hypersurface $\Sigma$ embedded in $M'$.

Next, we will give a new proof of the isoperimetric inequality  when the conformal vector field $\xi$ is closed. The proof is much simpler than before and we can release the condition $\varphi^2-\xi(\varphi)>0$ to $\varphi^2-\xi(\varphi)\geq 0$.

We simplify the evolution equations of $\vert\xi\vert^2$ and the support function $u$ first.
\begin{proposition}\label{evolutionofradial}
	Let $(M^{n+1},\bar{g})$ be a Riemannian manifold endowed with a closed, complete conformal vector field $\xi$. Let $\Sigma_0$ be a closed, star-shaped hypersurface embedded in $M'$, under the flow (\ref{flow}) we have
	\begin{align}
		\partial_t\vert\xi\vert^2&=u\Delta_g\vert\xi\vert^2-2\xi(\varphi)(\vert\xi\vert^2-u^2)\frac{u}{\vert\xi\vert^2}, \label{radialfun1}
	\end{align}
\end{proposition}
\begin{proof}
		By (\ref{iden1}) and (\ref{iden4}), we have 
		\begin{align}\label{supportfun1-eq0}
			\nabla\vert\xi\vert^2=2\varphi\xi^{\top}, 
		\end{align}
and
	\begin{align}\label{supportfun1-eq1}
		\nabla\varphi=\frac{\xi(\varphi)}{\vert\xi\vert^2}\xi^{\top},\ \nu(\varphi)=\frac{u}{\vert\xi\vert^2}\xi(\varphi).
	\end{align}
	Plugging these into (\ref{norm}), then (\ref{radialfun1}) follows from  (\ref{ccurvature}) and (\ref{ricci2}).
\end{proof}

Applying the evolution equation (\ref{radialfun1}), we now can obtain the uniform bound of $\vert\xi\vert^2$ along the flow (\ref{flow}).
\begin{proposition}\label{C0est}
	Let $(M^{n+1},\bar{g})$ be a Riemannian manifold endowed with a closed, complete conformal vector field $\xi$,  and assume that $\varphi>0$ on $M$. Let $\Sigma_0$ be a closed, star-shaped hypersurface embedded in $M'$, and $\Sigma(t)$ be a smooth solution of the flow (\ref{flow}) starting from $\Sigma_0$.
	Then for any $(p,t)\in\Sigma\times[0,T)$ we have
	\begin{align}\label{C0est1}
		\min_{p\in\Sigma}\vert\xi\vert^2(p,0)\leq \vert\xi\vert^2(p,t)\leq \max_{p\in\Sigma}\vert\xi\vert^2(p,0). 
	\end{align}
\end{proposition}
\begin{proof}
	Recall (\ref{supportfun1-eq0}), we have 
	\begin{align*}
		\vert\nabla\vert\xi\vert^2\vert^2=4\varphi^2(\vert\xi\vert^2-u^2).
	\end{align*}
	Since $\varphi>0$, at the critical point of $\vert\xi\vert^2$, the following condition holds:
	\begin{align*}
		\vert\xi\vert^2-u^2=0.
	\end{align*}
	Hence, at the critical point, by (\ref{radialfun1})  we have 
	\begin{align*}
		\partial_t\vert\xi\vert^2&=u\Delta_g\vert\xi\vert^2.
	\end{align*}
	 This implies (\ref{C0est1}) by the standard maximum principle.
\end{proof}
\begin{proposition}
Let $(M^{n+1},\bar{g})$ be a Riemannian manifold endowed with a closed, complete conformal vector field $\xi$. Let $\Sigma_0$ be a closed, star-shaped hypersurface embedded in $M'$,	under the flow (\ref{flow}), the support function $u$ evolves by
	\begin{align}\label{supportfun1}
		(\partial_t-u\Delta_g)u&=H\langle\xi,\nabla u\rangle+\vert A\vert^2u^2-2\varphi Hu-(\overline{Ric}(\mathcal{N},\mathcal{N})-\overline{Ric}(\nu,\nu))u^2\nonumber\\&\quad  +n\varphi^2-n\xi(\varphi)(1-\frac{u^2}{\vert\xi\vert^2}).
	\end{align}
\end{proposition}
\begin{proof}
	The proof follows from substituting (\ref{ricci2}) and  (\ref{supportfun1-eq1}) into (\ref{ev-support fun}).
\end{proof}

Now, we are prepared to control the gradient of $\vert\xi\vert^2$ along the flow.
 We consider the auxiliary function
\begin{align*}
	\omega:=\frac{1}{2}\vert\xi\vert^2\eta=\frac{1}{2}\vert\xi\vert^2(\frac{\vert\xi\vert^2}{u^2}-1).
\end{align*}
\begin{lemma}
	Let $(M^{n+1},\bar{g})$ be a Riemannian manifold endowed with a closed, complete conformal vector field $\xi$. Let $\Sigma_0$ be a closed, star-shaped hypersurface embedded in $M'$,	under the flow (\ref{flow}), the function $\omega$ evolves by
\begin{align}\label{auxifun2-eq0}
	\partial_t\omega=&u\Delta_g\omega+\frac{u^3}{\vert\xi\vert^4}(2\frac{\vert\xi\vert^2}{u^2}-3)\langle\nabla\omega,\nabla\vert\xi\vert^2\rangle-3\frac{u^3}{\vert\xi\vert^4}\vert\nabla\omega\vert^2+H\langle\nabla\omega,\xi\rangle\nonumber\\
 &-\frac{\vert\xi\vert^4}{u}\vert A\vert^2+(3\vert\xi\vert^2-u^2)\varphi H-2(\varphi^2-\xi(\varphi))\left(n-2\frac{u^2}{\vert\xi\vert^2}+\frac{u^4}{\vert\xi\vert^4}\right)\frac{\omega}{u}\\
	&-\varphi^2\left((n+2)\frac{\vert\xi\vert^2}{u}-7u+8\frac{u^3}{\vert\xi\vert^2}-3\frac{u^5}{\vert\xi\vert^4}\right)\nonumber\\
 &+\frac{\vert\xi\vert^4}{u}(\overline{Ric}(\mathcal{N},\mathcal{N})-\overline{Ric}(\nu,\nu)).\nonumber
\end{align}
\end{lemma}
\begin{proof}
	Combining (\ref{radialfun1}) with (\ref{supportfun1}) we have
	\begin{align}\label{auxifun2-eq1}
		\partial_t\omega=&(\frac{\vert\xi\vert^2}{u^2}-\frac{1}{2})\partial_t\vert\xi\vert^2-\frac{\vert\xi\vert^4}{u^3}\partial_tu\nonumber\\
		=&(\frac{\vert\xi\vert^2}{u^2}-\frac{1}{2})\left(u\Delta_g\vert\xi\vert^2-2\xi(\varphi)(\vert\xi\vert^2-u^2)\frac{u}{\vert\xi\vert^2}\right)-\frac{\vert\xi\vert^4}{u^3}\Big(u\Delta_gu\\
  &+H\langle\xi,\nabla u\rangle+\vert A\vert^2u^2-2\varphi Hu-(\overline{Ric}(\mathcal{N},\mathcal{N})-\overline{Ric}(\nu,\nu))u^2 \nonumber\\
  &+n\varphi^2-n\xi(\varphi)(1-\frac{u^2}{\vert\xi\vert^2})\Big).\nonumber
	\end{align}
Next, we compute
\begin{align}\label{auxifun2-eq2}
	\nabla\omega=&(\frac{\vert\xi\vert^2}{u^2}-\frac{1}{2})\nabla\vert\xi\vert^2-\frac{\vert\xi\vert^4}{u^3}\nabla u.
\end{align}
Thus
\begin{align}\label{auxifun2-eq3}
	\Delta_g\omega=&(\frac{\vert\xi\vert^2}{u^2}-\frac{1}{2})\Delta_g\vert\xi\vert^2-\frac{\vert\xi\vert^4}{u^3}\Delta_gu-4\frac{\vert\xi\vert^2}{u^3}\langle\nabla\vert\xi\vert^2,\nabla u\rangle\nonumber\\
 &+\frac{\vert\nabla\vert\xi\vert^2\vert^2}{u^2}+3\frac{\vert\xi\vert^4}{u^4}\vert\nabla u\vert^2\nonumber\\
	=&(\frac{\vert\xi\vert^2}{u^2}-\frac{1}{2})\Delta_g\vert\xi\vert^2-\frac{\vert\xi\vert^4}{u^3}\Delta_gu-\frac{u^2}{\vert\xi\vert^4}(2\frac{\vert\xi\vert^2}{u^2}-3)\langle\nabla\omega,\nabla\vert\xi\vert^2\rangle\\
	&+3\frac{u^2}{\vert\xi\vert^4}\vert\nabla\omega\vert^2-\frac{u^2}{\vert\xi\vert^4}(\frac{\vert\xi\vert^2}{u^2}-\frac{3}{4})\vert\nabla\vert\xi\vert^2\vert^2,\nonumber
\end{align}
where in the second equality we used (\ref{auxifun2-eq2}). Now, substituting (\ref{auxifun2-eq3}) into (\ref{auxifun2-eq1}) and using (\ref{auxifun2-eq2}) again, we see that
\begin{align}\label{auxifun2-eq4}
	\partial_t\omega=&u\Delta_g\omega+\frac{u^3}{\vert\xi\vert^4}(2\frac{\vert\xi\vert^2}{u^2}-3)\langle\nabla\omega,\nabla\vert\xi\vert^2\rangle-3\frac{u^3}{\vert\xi\vert^4}\vert\nabla\omega\vert^2+H\langle\nabla\omega,\xi\rangle\nonumber\\
	&+\frac{u^3}{\vert\xi\vert^4}(\frac{\vert\xi\vert^2}{u^2}-\frac{3}{4})\vert\nabla\vert\xi\vert^2\vert^2-H(\frac{\vert\xi\vert^2}{u^2}-\frac{1}{2})\langle\xi,\nabla\vert\xi\vert^2\rangle+2\frac{\vert\xi\vert^4}{u^2}\varphi H\\
	&-\frac{\vert\xi\vert^4}{u}\vert A\vert^2-n\varphi^2\frac{\vert\xi\vert^4}{u^3}+\xi(\varphi)\left(n\frac{\vert\xi\vert^4}{u^3}-(n+2)\frac{\vert\xi\vert^2}{u}+3u-\frac{u^3}{\vert\xi\vert^2}\right)\nonumber\\	&+\frac{\vert\xi\vert^4}{u}(\overline{Ric}(\mathcal{N},\mathcal{N})-\overline{Ric}(\nu,\nu)).\nonumber
\end{align}
Then, Plugging (\ref{supportfun1-eq0})  into  (\ref{auxifun2-eq4}) leads to (\ref{auxifun2-eq0}).
\end{proof}
\begin{proposition}\label{C1-est-cc}
		Let $(M^{n+1},\bar{g})$ be a Riemannian manifold endowed with a closed, complete conformal vector field $\xi$. We assume that $\varphi>0$, $\varphi^2-\xi(\varphi)\geq 0$ and the direction determined by $\xi$ is of least Ricci curvature on $M$. Let $\Sigma_0$ be a closed, star-shaped hypersurface embedded in $M'$, then along the flow (\ref{flow}),
		\begin{align*}
			\max_{\Sigma_t}\omega\leq \frac{C_1}{\sqrt{t+1}},
		\end{align*}
		where $C_1$ is a uniform constant depending on the positive lower bound of $\varphi$, uniform estimate of $\vert\xi\vert^2$ and $\max_{\Sigma_0}\omega$.
\end{proposition}
\begin{proof}
	Recall the evolution equation (\ref{auxifun2-eq0}) of the function $\omega$, since $\frac{u^2}{\vert\xi\vert^2}\leq 1$, by the assumption of Ricci curvature on $M$ and $\varphi^2-\xi(\varphi)\geq 0$ we have 
	\begin{align*}
	\partial_t\omega\leq&u\Delta_g\omega+\frac{u^3}{\vert\xi\vert^4}(2\frac{\vert\xi\vert^2}{u^2}-3)\langle\nabla\omega,\nabla\vert\xi\vert^2\rangle-\frac{u^3}{\vert\xi\vert^4}\vert\nabla\omega\vert^2+H\langle\nabla\omega,\xi\rangle\nonumber\\
	&+(3\vert\xi\vert^2-u^2)\varphi H-\frac{\vert\xi\vert^4}{u}\vert A\vert^2-\varphi^2\left((n+2)\frac{\vert\xi\vert^2}{u}-7u+8\frac{u^3}{\vert\xi\vert^2}-3\frac{u^5}{\vert\xi\vert^4}\right).\nonumber
	\end{align*}
We will work at a maximum point $p$ of the test function $\omega$, so that the following identity holds by (\ref{auxifun2-eq2}):
\begin{align}\label{cconformal gradient-eq1}
	(\frac{\vert\xi\vert^2}{u^2}-\frac{1}{2})\nabla\vert\xi\vert^2=\frac{\vert\xi\vert^4}{u^3}\nabla u.
\end{align}
We can pick the orthonormal frame $\{e_1=\frac{\xi^{\top}}{\vert\xi^{\top}\vert},e_2,\cdots,e_n\}$ around the point $p$, then (\ref{cconformal gradient-eq1}) leads to
\begin{align*}
	(2\frac{\vert\xi\vert^2}{u^2}-1)\varphi\langle\xi,e_i\rangle
	=\frac{\vert\xi\vert^4}{u^3}h_{ik}\langle\xi,e_k\rangle, \ \forall i=1,\cdots,n,
\end{align*}
where we used (\ref{cconformal}) to compute $\nabla_iu=h_{ik}\langle\xi,e_k\rangle$. Then, we have
\begin{align}
	h_{1i}\langle\xi,e_1\rangle=(2\frac{u}{\vert\xi\vert^2}-\frac{u^3}{\vert\xi\vert^4})\varphi\langle\xi,e_i\rangle,
\end{align}
this implies $h_{11}=\varphi(2\frac{u}{\vert\xi\vert^2}-\frac{u^3}{\vert\xi\vert^4})$ and $h_{1i}=0$ for any $i\neq 1$. Therefore, at the point $p$, $h_{11}$ is one of the principle curvatures and we denote the others by $\{\lambda_2,\cdots,\lambda_n\}$. Thus
\begin{align}\label{cconformal gradient-eq2}
		\partial_t\omega\leq&(3\vert\xi\vert^2-u^2)\varphi \sum_{i=2}^{n}\lambda_i-\frac{\vert\xi\vert^4}{u}\sum_{i=2}^{n}\lambda_i^2+(3\vert\xi\vert^2-u^2)\varphi h_{11}-\frac{\vert\xi\vert^4}{u}h_{11}^2\nonumber\\
		&-\varphi^2\left((n+2)\frac{\vert\xi\vert^2}{u}-7u+8\frac{u^3}{\vert\xi\vert^2}-3\frac{u^5}{\vert\xi\vert^4}\right)\nonumber\\
		=&-\frac{\vert\xi\vert^4}{u}\sum_{i=2}^{n}\left(\lambda_i-\frac{\varphi}{2}(3\frac{u}{\vert\xi\vert^2}-\frac{u^3}{\vert\xi\vert^4})\right)^2\nonumber\\
		&-\varphi^2\left((n+2)\frac{\vert\xi\vert^2}{u}-9u+9\frac{u^3}{\vert\xi\vert^2}-3\frac{u^5}{\vert\xi\vert^4}\right)\nonumber\\
  &+\frac{n-1}{4}\left(9u-6\frac{u^3}{\vert\xi\vert^2}+\frac{u^5}{\vert\xi\vert^4}\right)\varphi^2.
\end{align}
It's easy to check that
\begin{align}\label{cconformal gradient-eq3}
	&-\varphi^2\left((n+2)\frac{\vert\xi\vert^2}{u}-9u+9\frac{u^3}{\vert\xi\vert^2}-3\frac{u^5}{\vert\xi\vert^4}\right)\nonumber\\
 &+\frac{n-1}{4}\left(9u-6\frac{u^3}{\vert\xi\vert^2}+\frac{u^5}{\vert\xi\vert^4}\right)\varphi^2\nonumber\\
	=&-4(n-1)\left(\frac{u^3}{\vert\xi\vert^2}-\frac{u^5}{4\vert\xi\vert^4}\right)\frac{\omega^2}{\vert\xi\vert^4}\varphi^2-12\left(\frac{u^3}{\vert\xi\vert^2}-\frac{u^5}{\vert\xi\vert^4}\right)\frac{\omega^2}{\vert\xi\vert^4}\varphi^2\nonumber\\
	\leq&-8(n+2)\varphi^2\frac{u^5}{\vert\xi\vert^{10}}\omega^3.
\end{align}
Combining (\ref{cconformal gradient-eq2}) with (\ref{cconformal gradient-eq3}) we have
\begin{align*}
	\partial_t\omega\leq&-8(n+2)\varphi^2\frac{u^5}{\vert\xi\vert^{10}}\omega^3.
\end{align*}
By the maximum principle, there is a uniform upper bound for $\omega$. Therefore, with the uniform bound of $\vert\xi\vert^2$, we now have 
\begin{align*}
	u\geq\epsilon_2>0,
\end{align*}
where $\epsilon_2$ is a uniform constant independent of the time $t$. Then, there is a uniform constant $K_3$ independent of $t$ such that 
\begin{align*}
	\partial_t\omega\leq&-K_3\omega^3.
\end{align*}
This implies that 
\begin{align*}
	\max_{\Sigma_t}\omega\leq \frac{C_1}{\sqrt{t+1}},
\end{align*}
where $C_1$ is a uniform constant depending on $K_3$ and $\max_{\Sigma_0}\omega$.
\end{proof}

Notice that around any point on the evolving hypersurface  along the flow, by choosing the normal coordinate $\{e_i\}_{i=1}^n$ we have
\begin{align}\label{cconvergence-eq1}
	\nabla_i\vert\xi\vert^2=2\varphi\langle\xi,e_i\rangle,
\end{align}
\begin{align}\label{cconvergence-eq2}
	\nabla_j\nabla_i\vert\xi\vert^2=2\frac{\xi(\varphi)}{\vert\xi\vert^2}\langle\xi,e_i\rangle\langle\xi,e_j\rangle+2\varphi^2 g_{ij}-2\varphi uh_{ij}.
\end{align}
 Since the evolution equation (\ref{radialfun1}) of the function $\vert\xi\vert^2$ is a quasilinear parabolic equation, it follows from Proposition \ref{C0est} and Proposition \ref{C1-est-cc} that the equation is uniform parabolic. Then the regularity of $\vert\xi\vert^2$ follows from the standard parabolic theory.  This implies  the regularity of the second fundamental form $h_{ij}$ and its higher oder derivatives by (\ref{cconvergence-eq2}). Hence, we obtain the long time existence and convergence of the flow (\ref{flow}). Moreover, the solution converges to a leaf of the foliation $\mathcal{F}(\xi)$ by Proposition \ref{C1-est-cc}
. Now, we have proved the following theorem.
\begin{theorem}\label{conergence-closed}
	Let $(M^{n+1},\bar{g})$ be a Riemannian manifold endowed with a closed, complete conformal vector field $\xi$. We assume that $\varphi>0$, $\varphi^2-\xi(\varphi)\geq 0$ and the direction determined by $\xi$ is of least Ricci curvature on $M$.  Let $\Sigma_0$ be a closed, star-shaped hypersurface embedded in $M'$. Then the evolution equation (\ref{flow}) with $\Sigma_0$ as a initial data has a smooth solution for $t\in[0,+\infty)$ and the solution hypersurfaces converge smoothly to a leaf of the foliation $\mathcal{F}(\xi)$.
\end{theorem}
Then, applying Theorem \ref{monotonethm}, we obtain the isoperimetric inequality in Theorem \ref{isoineq3}.

\bibliography{isoperimetric_ineq}
\end{document}